\documentclass[11pt,reqno]{amsart}
\usepackage[utf8]{inputenc}
\usepackage[margin=1.2in, centering]{geometry}
\usepackage{amsmath,amssymb,latexsym,mathtools,comment}

\usepackage[shortlabels]{enumitem}

\usepackage{amsthm}
\newtheorem{thm}{Theorem}[section]
\newtheorem{lem}[thm]{Lemma}

\newtheorem{prop}[thm]{Proposition}
\newtheorem{problem}[thm]{Problem}

\numberwithin{equation}{section}
\theoremstyle{definition}

\newtheorem{defn}[thm]{Definition}

\newtheorem{remark}[thm]{Remark}

\newtheorem*{nota*}{Notation}

\usepackage{tikz}
\usepackage{graphicx}
\DeclareGraphicsExtensions{.pdf,.png,.jpg}
\usepackage{subfig}

\newcommand{\bN}{\mathbb{N}}
\newcommand{\bC}{\mathbb{C}}

\newcommand{\bR}{\mathbb{R}}

\newcommand{\bZ}{\mathbb{Z}}
\newcommand{\bT}{\mathbb{T}}

\newcommand{\RR}{\mathbb{R}}
\newcommand{\ZZ}{\mathbb{Z}}
\newcommand{\NN}{\mathbb{N}}
\newcommand{\TT}{\mathbb{T}}
\def\forinfmany{\text{ for infinitely many }}

\newcommand{\rbr}[1]{\left( {#1} \right)}

\newcommand{\cbr}[1]{\left\{ {#1} \right\}}

\newcommand{\abs}[1]{\left| {#1} \right|}

\newcommand{\norm}[1]{{\left\Vert#1\right\Vert}}

\DeclareMathOperator*{\supp}{supp}

\usepackage{hyperref}

\begin{document}
\title{Fourier restriction and well-approximable numbers}
\date{}

\author{Robert Fraser}
\address{Department of Mathematics and Statistics, Wichita State University, Wichita, KS, USA}
	\email{robert.fraser@wichita.edu}

\author{Kyle Hambrook}
\address{Department of Mathematics and Statistics, San Jose State University, San Jose, CA, USA}
	\email{kyle.hambrook@sjsu.edu}

\author{Donggeun Ryou}
\address{Department of Mathematics, Indiana University, Bloomington, IN, USA}
	\email{dryou@iu.edu}

\begin{abstract}
We use a deterministic construction to prove the optimality of the exponent in the Mockenhaupt-Mitsis-Bak-Seeger Fourier restriction theorem for dimension $d=1$ and parameter range $0 < a,b \leq d$ and $b\leq 2a$. Previous constructions by Hambrook and {\L}aba \cite{HL2013} and Chen \cite{chen} required randomness and only covered the range  $0 < b \leq a \leq d=1$. We also resolve a question of Seeger \cite{seeger-private} about the Fourier restriction inequality on the sets of well-approximable numbers.  
\end{abstract}

\maketitle

\section{Introduction}


\subsection{Statement of the problem}

This paper concerns the optimality of the range of exponents $p$ in the following $L^2$ Fourier restriction theorem.  

\begin{thm}[Mockenhaupt-Mitsis-Bak-Seeger]\label{frac ST}
Let $0 < a, b < d$, with $d$ an integer.
Define $p_{*}(a,b,d) = (4d-4a+2b)/b.$
Suppose $\mu$ is a Borel probability measure on $\RR^d$ with the following properties: 
\begin{enumerate}[(A)]
    \item $\mu(B(x,r)) \lesssim r^{a}$ for all $x \in \RR^d$ and all $r > 0$.  
    \item $|\mathcal{F}(\mu)(\xi)| \lesssim (1+|\xi|)^{-b/2}$ for all $\xi \in \RR^d$.  
\end{enumerate}
Then, for each $p \geq p_{*}(a,b,d)$, the following holds: 
\begin{enumerate}[(A)]
\setcounter{enumi}{17}

\item $\|\mathcal{F}(f\mu)\|_{L^p(\lambda)} \lesssim_p \|f\|_{L^2(\mu)}$ for all $f \in L^2(\mu)$.
\end{enumerate}

\end{thm}

\begin{nota*} 
$\lambda$ is Lebesgue measure on $\RR^d$ and $B(x,r)$ is the open ball with center $x$ and radius $r$. A statement of the form ``$A(x) \lesssim B(x)$ for all $x$'' means ``there exists a constant $C>0$ such that $A(x) \leq C B(x)$ for all $x$.'' If the implied constant $C$ depends on a relevant parameter, say $\epsilon$, we write $\lesssim_{\epsilon}$ instead of $\lesssim$. The usage of $\gtrsim$ is analogous. 
$A(x) \sim B(x)$ means $A(x) \lesssim B(x)$ and $B(x) \lesssim A(x)$. 
%
For a measure $\mu$ on $\RR^d$, the Fourier transform of $\mu$ is given by $\mathcal{F}(\mu)(\xi) = \int_{\RR^d} e^{-2\pi i x \cdot \xi} d\mu(x)$. 
We denote $[-1/2,1/2]$ by $\bT$. 
If the support of $\mu$ is contained in $\bT^d$, the Fourier coefficients are given by  
$\widehat{\mu}(s) = \int_{\bT^d} e^{-2\pi i x \cdot s} d\mu(x)$. For an interval $I $ in $\bR$, we denote the Lebesgue measure of $I$ by $|I|$. 
\end{nota*}

Mockenhaupt \cite{mock} and Mitsis \cite{mitsis-res} independently proved Theorem \ref{frac ST} with $p > p_{*}(a,b,d)$;  
the endpoint case $p = p_{*}(a,b,d)$ was established by Bak and Seeger \cite{bak-seeger}. 
Theorem \ref{frac ST} is a generalization of the classic Stein-Tomas Theorem \cite{tomas-1975}, \cite{tomas-symp}, 
which can be understood as the special case of Theorem \ref{frac ST} in which $\mu$ is the surface measure on a sphere in $\RR^d$, $d \geq 2$. Generalizations of the Stein-Tomas Theorem to quadratic hypersurfaces and arbitrary smooth submanifolds in $\RR^d$ were obtained, respectively, by Strichartz \cite{Strichartz77} and Greenleaf \cite{Greenleaf}; these results are also consequences of Theorem \ref{frac ST}.

The following problem was suggested in the original papers of Mockenhaupt \cite{mock} and Mitsis \cite{mitsis-res}.

\begin{problem}\label{prob}
Given $a,b,d$, determine whether the value $p_{*}(a,b,d)$ in Theorem \ref{frac ST} is optimal. 
\end{problem}
 
To prove that $p_{*}(a,b,d)$ is optimal for a fixed triple $(a,b,d)$, we must prove the following: 
For each $p_0 < p_*(a,b,d)$, 
there exists a Borel probability measure $\mu$ on $\RR^d$ that satisfies (A) and (B) and there exists $p \geq p_0$ such that (R) is false.

\begin{remark}
There are constructions of measures which satisfy (A), (B), and (R) for $p$ outside the range $p \geq p_{*}(a,b,d)$;  see Shmerkin and Suomala \cite{shmerkin-suomala}, Chen and Seeger \cite{chen-seeger}, {\L}aba and Wang \cite{laba-wang}, Ryou \cite{Ryou-2023}. But this is a different problem than the one we consider here. 
\end{remark}

\subsection{Motivation}

As motivation for our main result, we start by discussing some counterexamples in the theory of restriction to curves and surfaces.

When $d \geq 2$ and $a=b=d-1$, $p_{*}(a,b,d)$ in Theorem \ref{frac ST} is known to be optimal because of a construction due (according to \cite{tomas-symp}) to Anthony W. Knapp. 
In Knapp's construction, $\mu$ is the surface measure on a sphere, which satisfies (A) and (B) with $a=b=d-1$ (hence $p_{*}(a,b,d) = (2d+2)/(d-1)$).  
Knapp's function $f$ is an indicator function of a sufficiently flat cap on the sphere. (To be more precise, $f$ is the indicator function of a rectangular region in $\RR^d$ with center at a point on the sphere, with one side of length $\delta$, and with $d-1$ sides of length $\sqrt{\delta}$. The region is oriented so that the side of length $\delta$ lies normal to the sphere. The estimate (R) is shown to fail by taking $\delta \to 0$.) 

In fact, Knapp's example can be modified to apply for \textit{any} smooth hypersurface. The following result appears as Lemma 3 of Strichartz \cite{Strichartz77} in the case $q = 2$. It appears as Example 1.8 of Demeter \cite{Demeter} for general $q$. 
\begin{thm}\label{generalknapp}
Let $1 \leq p, q \leq \infty.$ Let $S$ be a smooth $(d-1)$-dimensional submanifold of $\mathbb{R}^d$ and let $\mu$ be any smooth Borel probability measure on $S$. Then the estimate
\[\|\mathcal{F}(f\mu)\|_{L^p(\lambda)} \lesssim_{p,q} \|f\|_{L^q(\mu)} \text{ for all $f \in L^q(\mu)$}\]
cannot hold if $p < \frac{q(d + 1)}{(q - 1) (d - 1)}.$
\end{thm}
Note that if $a = b = d-1$ and $q = 2$, this theorem recovers the exponent $p_*(a, b, d)$ from Theorem \ref{frac ST}. Therefore, for the surface measure of a smooth hypersurface with non-vanishing Gaussian curvature, Theorem \ref{frac ST} is sharp. 

Greenleaf \cite[Remark 1.3]{Greenleaf} observed that Knapp's example can be generalized to apply to any smooth hypersurface $S$ given as the Cartesian product of a $k$-dimensional surface and $(d-1-k)$-dimensional Euclidean space. In particular, if the $k$-dimensional surface has non-zero Gaussian curvature and if $\mu$ is the surface measure on $S$, this modified Knapp example shows that $p_*(a, b, d)$ in Theorem \ref{frac ST} is optimal when $a=d-1$ and $b=k$.

A measure $\mu$ satisfying (A) can only satisfy (R) if $p \geq 2d/a$. 
For a simple proof of this result, see Mitsis \cite[Proposition 3.1]{mitsis-res}. 
This result does not tell us anything directly about Problem 1.2.
However, by combining it with Theorem \ref{frac ST}, it implies that if a measure satisfies (A) and (B), then $b \leq 2a$. 

Ryou \cite{Ryou-2023} builds on the $p \geq 2d/a$ result mentioned above 
using a Knapp-type example.     
In this way, Ryou establishes 
the following result for fractal subsets of the paraboloid (which appears as Theorem 1.3 in \cite{Ryou-2023}.)
\begin{thm}\label{ryou}
Let $1 \leq p, q \leq \infty.$ Let $\mathbb{P}^{d-1}$ denote the $(d-1)$-dimensional submanifold in $\mathbb{R}^d$ paramaterized by $(x, |x|^2) : x \in \mathbb{R}^{d-1}$, and assume $\mu$ is a Borel probability measure supported on a subset of $\mathbb{P}^{d-1}$ of Hausdorff dimension $a \in (0,d-1)$ and satisfying the estimate
\[\mu(B(x, r)) \lesssim r^{a}\]
for all $x \in \mathbb{P}^{d-1}$ and all $r > 0$. Then the estimate
\[\|\mathcal{F}(f\mu)\|_{L^p(\lambda)} \lesssim_{p,q} \|f\|_{L^q(\mu)} \text{ for all $f \in L^q(\mu)$}\]
cannot hold if $p < \frac{q(d + 1)}{(q - 1) a}$.
\end{thm}
Note that this result does not tell us about whether $p_*(a,b,d)$ is optimal in Theorem \ref{frac ST}. However, by combining Theorem \ref{frac ST} and Theorem \ref{ryou} when $q=2$, it follows that a measure $\mu$ that satisfies the hypothesis of Theorem \ref{ryou} can only satisfy (B) with $b \leq 2a(d - a) / (1+d-a)$.    

Restriction problems have also been considered for manifolds of codimension larger than $1$. A classic example is the moment curve $\Gamma \subseteq  \mathbb{R}^d$ parameterized by $(x, x^2, x^3, \ldots, x^d)$. The optimal extension estimates for this curve were obtained by Drury \cite{Drury}:
\begin{thm}\label{drury}
Let $\mu$ denote the arc-length measure on the moment curve $\Gamma$. Then the estimate 
\[\|\mathcal{F}(f\mu)\|_{L^p(\lambda)} \lesssim_p \|f\|_{L^q(\mu)} \text{ for all $f \in L^q(\mu)$}\]
holds if $p > \frac{d(d+1)}{2} + 1$ and $q = \frac{2p}{2p - d(d+1)}$.
\end{thm}
The endpoint $d(d+1)/2 + 1$ is known to be sharp; an argument of Arhipov, Cubarikov, and Karacuba \cite{ACK} shows that Theorem \ref{drury} is false if $p = \frac{d(d+1)}{2} + 1$. 
However, this argument does not tell us whether $p_*(a,b,d)$ is optimal in Theorem \ref{frac ST}. 
Indeed, when we consider the arc-length measure on the moment curve, $a=1$, $b=1/d$, and $p_*(a,b,d) = 4d(d-1)+2$, which is strictly larger than $d(d+1)/2 + 1$. 

We now turn to past results on Problem \ref{prob} when $a$ and $b$ are not necessarily integers. 
Hambrook and {\L}aba \cite{HL2013} (see also Hambrook \cite{hambrook-thesis}) proved the optimality of $p_*(a, b, d)$ in Theorem \ref{frac ST} when $d=1$ and $0 < b \leq a \leq d$. 
Since there is no notion of curvature in one dimension, the construction takes a form completely different from Knapp's. The authors modified a construction of {\L}aba and Pramanik \cite{LP} to construct a random Cantor set containing 
a small deterministic Cantor set determined by 
a multi-scale arithmetic progression;  
more details can be found in Subsection \ref{idea} below.  
The construction was improved in technical ways by Chen \cite{chen}. 
In \cite{HL2016}, Hambrook and {\L}aba proved the optimality of $p_*(a,b,d)$ for $d \geq 2$ and $d-1 \leq b \leq a \leq d$. 
The construction combines the examples of Knapp and \cite{HL2013}.

Seeger \cite{seeger-private} asked whether it is possible to prove optimality of Theorem \ref{frac ST} when $d=1$ with a \textit{deterministic} construction. 
In particular, Seeger wondered whether Theorem \ref{frac ST} is optimal for measures on the following sets:
$$
E(\alpha) = \cbr{x \in \RR : |x-r/q| \leq |q|^{-(2+\alpha)} \forinfmany (q,r) \in \ZZ \times \ZZ },
$$
where $\alpha > 0$. The set $E(\alpha)$ is called the set of $\alpha$-well-approximable numbers. 
These sets arise from number theory (in particular, Diophantine approximation). 
They are important in harmonic analysis because they are the only known \textit{deterministic} examples of so-called Salem sets in $\RR$. 

A set in $\RR^d$ is called a Salem set if it supports a probability measure $\mu$ that satisfies (B) for every value of $b$ less than the Hausdorff dimension of the set. 
There are many random constructions of Salem sets in $\RR^d$ (see \cite{salem}, \cite{kahane-book}, \cite{kahane-1966-fourier}, \cite{kahane-1966-brownian}, \cite{bluhm-1}, \cite{chen-seeger}, \cite{ekstrom}, \cite{LP}, \cite{shmerkin-suomala}). 
Trivial examples of deterministic Salem sets in $\RR^d$ are points (Hausdorff dimension 0), spheres (Hausdorff dimension $d-1$), and balls (Hausdorff dimension $d$). 
Kaufman \cite{Kaufman} proved that $E(\alpha)$ is a Salem set of dimension $2/(2+\alpha)$ when $\alpha > 0$.
The sets $E(\alpha)$ are the only known deterministic examples of Salem sets in $\RR$ of dimension other than 0 or 1.

\subsection{Main Result}

In the present paper, 
we prove that $p_*(a,b,d)$ is optimal for $d=1$, $0 < a,b < d$ and $b \leq 2a$ with a completely \textit{deterministic} construction. 
As we mentioned above, $b>2a$ cannot happen. Thus we cover all possible $a$ and $b$ when $d=1$. Recall that previous work only covered the case $b \leq a$. We also resolve the question of Seeger mentioned above. Our precise result is contained in the following two theorems.

\begin{thm}\label{main-result}
Let $\alpha > 0$ and $0 \leq \beta < 1$. There exists a Borel probability measure $\mu$ such that the following hold: 
\begin{itemize}
\item The support of $\mu$ is contained in $[-1/2, 1/2] \cap E(\alpha)$. 
\item For any $\epsilon >0$, 
\begin{equation}\label{MT_reg_mu}
\mu(I) \lesssim_\epsilon |I|^{\frac{2}{2+\alpha}-\epsilon} \quad \text{for every interval $I \subseteq \mathbb{R}$}. 
\end{equation}
\item For any $\epsilon >0$, 
\begin{equation}\label{MT_FD_mu}
|\mathcal{F}(\mu)(\xi)| \lesssim_\epsilon (1+|\xi|)^{-\frac{1-\beta}{2+\alpha}+\epsilon} \quad \text{for every $\xi \in \RR$}. 
\end{equation}
\item There exists a sequence of nonnegative functions $\{f_k\}_{k \in \bN}$ such that
\begin{equation}\label{LqLpest2}
\sup_{k \geq 1} \frac{\norm{\mathcal{{F}}(f_k \mu)}_{L^p(\lambda)}}{\norm{f_k}_{L^q(\mu)}} =\infty \qquad \text{whenever}  \ p< p_+(q) := \frac{q}{q-1}\left(\frac{1  + \alpha - \beta}{1-\beta}\right).
\end{equation}
\end{itemize}
\end{thm}
\begin{thm}\label{main-result2}
Let $\alpha > 0$ and $-1 < \beta <0$. 
There exists a Borel probability measure $\mu$ such that the following hold: 
\begin{itemize}
\item The support of $\mu$ is contained in $[-1/2, 1/2] \cap E(\alpha)$. 
\item For any $\epsilon >0$, 
\begin{equation}\label{MT_reg_mu_2}
\mu(I) \lesssim_\epsilon |I|^{\frac{2+\beta}{2+\alpha}-\epsilon} \quad \text{for every interval $I \subseteq \mathbb{R}$}. 
\end{equation}
\item For any $\epsilon >0$, 
\begin{equation}\label{MT_FD_mu_2}
|\mathcal{F}(\mu)(\xi)| \lesssim_\epsilon (1+|\xi|)^{-\frac{1}{2+\alpha}+\epsilon} \quad \text{for every $\xi \in \RR$}. 
\end{equation}
\item There exists a sequence of nonnegative functions $\{f_k\}_{k \in \bN}$ such that
\begin{equation}\label{LqLpest2_2}
\sup_{k \geq 1} \frac{\norm{\mathcal{{F}}(f_k \mu)}_{L^p(\lambda)}}{\norm{f_k}_{L^q(\mu)}} =\infty \qquad \text{whenever}  \ p< p_-(q) := \frac{q}{q-1}\left(1+\alpha-\beta\right).
\end{equation}
\end{itemize}
\end{thm}

\begin{remark} Theorem \ref{main-result} and \ref{main-result2} imply the optimality of $p_{*}(a,b,d)$ in Theorem \ref{frac ST} when $0 < a,b < d=1$ and $b \leq 2a$. This is a consequence of the following argument. 
Let $p_0 < p_*(a,b,d)$ be given. 
Choose $\epsilon > 0$ such that $p_0 < p_*(a+\epsilon,b+\epsilon,d) < p_*(a,b,d)$.

If $b\leq a$, choose $\alpha > 0$ and $0\leq  \beta < 1$ such that 
\begin{equation*}
a + \epsilon = \frac{2}{2+\alpha}, \qquad b + \epsilon = \frac{2(1-\beta)}{2+\alpha}.
\end{equation*}
Set $q = 2$. Then 
\begin{equation*}
p_+(2) = 2\left(\frac{1+\alpha-\beta}{1-\beta} \right) = \frac{4-4a+2b-2\epsilon}{b+\epsilon} = p_{*}(a+\epsilon,b+\epsilon,d) . 
\end{equation*}

If $a< b \leq 2a$, choose $\alpha > 0$ and $-1 < \beta < 0$ such that 
\begin{equation*}
a + \epsilon = \frac{2+\beta}{2+\alpha}, \qquad b + \epsilon = \frac{2}{2+\alpha}.
\end{equation*}
Set $q = 2$. Then 
\begin{equation*}
p_-(2) = 2(1+\alpha-\beta) = \frac{4-4a+2b-2\epsilon}{b+\epsilon} = p_{*}(a+\epsilon,b+\epsilon,d) . 
\end{equation*}
\end{remark}

\begin{figure}
    \centering

\tikzset{every picture/.style={line width=0.75pt}} 
\begin{tikzpicture}[x=0.75pt,y=0.75pt,yscale=-1,xscale=1]

\draw  (201.33,228.3) -- (430.56,228.3)(224.26,22) -- (224.26,251.22) (423.56,223.3) -- (430.56,228.3) -- (423.56,233.3) (219.26,29) -- (224.26,22) -- (229.26,29)  ;
\draw  [dash pattern={on 4.5pt off 4.5pt}]  (321.22,48.16) -- (224.26,228.3) ;
\draw  [dash pattern={on 4.5pt off 4.5pt}]  (321.22,48.16) -- (411.22,48.16) ;
\draw  [dash pattern={on 4.5pt off 4.5pt}]  (411.22,48.16) -- (224.26,228.3) ;
\draw  [dash pattern={on 4.5pt off 4.5pt}]  (411.22,48.16) -- (411.22,228.16) ;
\draw    (338.56,119.49) -- (338.56,226.16) ;
\draw [shift={(338.56,228.16)}, rotate = 270] [color={rgb, 255:red, 0; green, 0; blue, 0 }  ][line width=0.75]    (10.93,-3.29) .. controls (6.95,-1.4) and (3.31,-0.3) .. (0,0) .. controls (3.31,0.3) and (6.95,1.4) .. (10.93,3.29)   ;
\draw [shift={(338.56,119.49)}, rotate = 90] [color={rgb, 255:red, 0; green, 0; blue, 0 }  ][fill={rgb, 255:red, 0; green, 0; blue, 0 }  ][line width=0.75]      (0, 0) circle [x radius= 3.35, y radius= 3.35]   ;
\draw    (338.56,119.49) -- (285.89,119.49) ;
\draw [shift={(283.89,119.49)}, rotate = 360] [color={rgb, 255:red, 0; green, 0; blue, 0 }  ][line width=0.75]    (10.93,-4.9) .. controls (6.95,-2.3) and (3.31,-0.67) .. (0,0) .. controls (3.31,0.67) and (6.95,2.3) .. (10.93,4.9)   ;

\draw (440.67,224.73) node [anchor=north west][inner sep=0.75pt]    {$a$};
\draw (202,21.44) node [anchor=north west][inner sep=0.75pt]    {$b$};
\draw (260,38.73) node [anchor=north west][inner sep=0.75pt]    {$b=2a$};
\draw (406,236.73) node [anchor=north west][inner sep=0.75pt]    {$1$};
\draw (420.67,38.73) node [anchor=north west][inner sep=0.75pt]    {$b=a$};
\draw (356,104.73) node [anchor=north west][inner sep=0.75pt]    {$\beta =0$};
\draw (320.67,236.73) node [anchor=north west][inner sep=0.75pt]    {$\beta \rightarrow 1$};
\draw (228.67,104.73) node [anchor=north west][inner sep=0.75pt]    {$\beta \rightarrow -1$};

\end{tikzpicture}

    \caption{If we ignore $\epsilon$ in the exponents, $(a,b) = (\frac{2}{2+\alpha}, \frac{2}{2+\alpha})$ when $\beta=0$. The point $(a,b) \rightarrow (\frac{2}{2+\alpha}, 0)$ as $\beta \rightarrow 1$ and $(a,b) \rightarrow (\frac{1}{2+\alpha}, \frac{2}{2+\alpha})$ as $\beta \rightarrow -1$.}
    \label{fig1}
\end{figure}
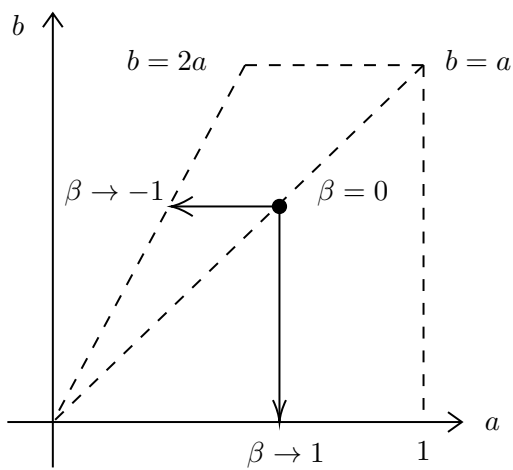

\subsection{Idea of the Proof of Theorem \ref{main-result} and Theorem \ref{main-result2}}\label{idea}

The proof of Theorem \ref{main-result} 
combines ideas from Kaufman \cite{Kaufman}, Papadimitropoulos \cite{pa-thesis}, and  Hambrook and {\L}aba \cite{HL2013}. We will first describe some key features of those works. Then will describe our proof of Theorem \ref{main-result} and \ref{main-result2}.

Kaufman \cite{Kaufman} constructs a measure on $E(\alpha)$ that obeys (B) with $b = 2/(2+\alpha)$ (here and elsewhere in this subsection, we ignore logarithmic factors and arbitrarily small exponent losses in (A) and (B)). However, as observed by Papadimitropoulos \cite{pa-thesis}, Kaufman's measure satisfies (A) with $a = 1/(2+\alpha)$ but not any larger value of $a$.  
Papadimitropoulos \cite{pa-thesis} modifies Kaufman's construction to obtain a measure that satisfies (A) and (B) with $a=b=2/(2+\alpha)$. 
Note that neither Kaufman nor Papadimitropoulos considered the inequality (R) for their measures. 
Note also that Kaufman's and Papadimitropoulos's measures are actually constructed on a Cantor set contained in $E(\alpha)$. As we describe below, this Cantor set plays an important role in our proof described below.

Hambrook and {\L}aba \cite{HL2013} start with a random Cantor set and then adjoin a deterministic Cantor set with arithmetic progression structure. 
More specifically, at each scale, they replace a small number of intervals of the random Cantor set with intervals of the same length but whose centers lie in a generalized arithmetic progression; these replacement intervals are the intervals of the structured deterministic Cantor set. 
The original random Cantor set is built so that its natural measure satisfies (A) and (B), for arbitrary prescribed values of $a$ and $b$ with $0 < b \leq a \leq 1$.
(In fact, Hambrook and {\L}aba \cite{HL2013} only obtained the case $a=b$.
Chen \cite{chen} and Hambrook \cite{hambrook-thesis} modified the construction to get $b \leq a$.) 
Because the number of intervals replaced is small enough, the natural measure on the new Cantor set (i.e., the Cantor set obtained by adjoining the deterministic Cantor set to the random one)  
satisfies (A) and (B) with the same values of $a$ and $b$. 
The measure on the new Cantor set is the measure 
that shows $p_*(a,b,d)$ is optimal for Theorem \ref{frac ST}. 
The functions $f_k$ which witness the failure of (R) are bump functions on the intervals in the $k$-th scale of the deterministic Cantor set. 
The arithmetic progression structure of the deterministic Cantor set is the key to obtaining the failure of (R).

For the proof of Theorem \ref{main-result} and Theorem \ref{main-result2}, 
we cannot simply \textit{add} an arithmetic progression-structured Cantor set to $E(\alpha)$ because we need the support of the measure $\mu$ to be inside $E(\alpha)$. 
Instead, we must \textit{identify} an appropriately structured set that is contained in $E(\alpha)$. 
We start with a Cantor set $\bigcap_{k \in \NN} E(\alpha,k)$ contained in $E(\alpha)$;  
this is essentially the same Cantor set that Kaufman's (and Papadimitropoulos's) measure was constructed on. The precise definition of $E(\alpha,k)$ is given in Section \ref{sec_measure}. For now, it is enough to know that $E(\alpha,k)$ is a union of intervals of length $M_k^{-(2+\alpha)}$, where $M_k$ is a rapidly increasing sequence of positive numbers. At the $k$-th scale, if $k$ is odd, we adjoin to $E(\alpha,k)$ a set $C(\beta,k)$  
which consists of a small number of intervals of length $ M_k^{-(2+\alpha)}$ 
whose centers form an arithmetic progression. 
(The precise definition of $C(\beta,k)$ is given in Section \ref{sec_measure}.)
For even $k$, we do not adjoin anything to $E(\alpha,k)$. 
We obtain a Cantor set $$\bigcap_{k \text{ even}} E(\alpha,k) \cap \bigcap_{k \text{ odd}} (E(\alpha,k) \cup C(\beta,k)).$$  
Because of the $\limsup$ structure of $E(\alpha)$, this Cantor set is contained in $E(\alpha)$. 
To construct our measure $\mu$ on 
this Cantor set, we adapt Papadimitropoulos's modification of Kaufman's construction. The number of intervals in $C(\beta,k)$ is chosen small enough so that the measure on the new Cantor 
set still satisfies (A) and (B).
(Notably, our proof of (A) differs substantially from that of Papadimitropoulos in that we have used a Fourier analytic argument whereas Papadimitropoulos uses a purely spatial argument.) 
The functions $f_k$ which witness the failure of (R) are bump functions on the intervals in $C(\beta,k)$. 
The parameter $\beta$ controls both the number of and the relative mass on the intervals 
in $C(\beta,k)$. 
When $\beta = 0$, the construction will produce a measure $\mu$ that satisfies (A) and (B) with $a=b=2/(2+\alpha)$ and fails (R) for the appropriate value of $p$. 
The $\beta = 0$ case is the closest to Papadimitropoulos and Kaufman's constructions (which also obtained $a=b=2/(2+\alpha)$). 
To get the case $a \neq b$, we take $\beta \neq 0$. 
When $\beta > 0$, the effect is (mostly) to increase the number of intervals at each scale of $C(\beta)$, which lowers the Fourier decay exponent $b$. 
When $\beta < 0$, the dominant effect is an increase of the relative mass on the intervals at each scale of $C(\beta)$, which decreases the regularity exponent $a$. 

\subsection{Structure of Remainder of Paper}

The rest of the paper is devoted to the proof of Theorem \ref{main-result} and Theorem \ref{main-result2}. Section \ref{secconv} states and proves the Convolution Stability Lemma, which is a fundamental tool in the verification of the various properties of the measure $\mu$. Section \ref{sec_measure} constructs the measure $\mu$, proves that it has support contained in $[-1/2,1/2] \cap E(\alpha)$, 
and proves that it satisfies the required regularity and Fourier decay estimates on $\mu$ (namely, \eqref{MT_reg_mu} and \eqref{MT_FD_mu} if $0 \leq \beta < 1$ and \eqref{MT_reg_mu_2} and \eqref{MT_FD_mu_2} if  $-1 < \beta < 0$). 
Section \ref{failure_sec} constructs the sequence of functions $f_k$ and 
verifies the failure of the inequality (R) for that sequence by proving \eqref{LqLpest2} if $0 \leq \beta < 1$ and \eqref{LqLpest2_2} if $-1 < \beta < 0$. 

\subsection{Acknowledgements} 

We thank Andreas Seeger for posing his problem to us and hence inspiring this paper. We thank the anonymous referee for their helpful comments and suggestions which improved the presentation of this paper. 

\section{Convolution Stability Lemma}\label{secconv}

In this section, we state and prove the Convolution Stability Lemma (Lemma \ref{convstab}). 
Before that, we present some discussion of it.

In the course of the proof of Theorem \ref{main-result} and \ref{main-result2}, we will construct a sequence of functions $g_k$ with Fourier transforms that are large at well-separated scales. We will obtain the measure $\mu$ as a weak limit of measures $\mu_k$ with densities given by
\[F_0(x) g_1(x) \ldots g_k(x),\]
where $F_0$ is a suitable cutoff function to be described later.
We will show that $\widehat{\mu}_{k-1}(s)$ is small unless $|s|$ lies below a certain threshold, and that $\widehat{g}_k(s)$ is equal to zero unless $|s|$ is larger than some other threshold. If these thresholds are sufficiently well-separated, there will be very little interference when we multiply $\mu_{k-1}(x)$ and $g_k(x)$ to arrive at $\mu_k(x)$. 
The Convolution Stability Lemma quantifies this. 

We will use the Convolution Stability Lemma to prove the 
required regularity and Fourier decay estimates on $\mu$ and to prove the failure of the inequality (R). 


Kaufman \cite{Kaufman} uses a lemma (the only lemma in that paper) that is similar to the Convolution Stability Lemma to establish the Fourier decay estimate on the measure constructed in that paper. 
The works \cite{bluhm-1}, \cite{pa-thesis}, \cite[Ch.9]{wolff-book}, \cite{hambrook-tams}, and \cite{hambrook-fraser-Rn} do likewise with similar lemmas.  
(Note that Papadimitropoulos \cite{pa-thesis} proves both a Fourier decay and regularity estimate on a measure. The Fourier decay estimate is established using a lemma similar to Kaufman's, but the regularity estimate is proved using a counting argument.) 
The Convolution Stability Lemma in the present paper differs from these other lemmas in 
that it gives different estimates at different frequency scales. 
This furnishes a more precise understanding of behaviour of the measures $\mu_k$ at different frequency scales and was helpful in formulating our proofs. 
Note that \cite{fraser-hambrook-p-adic} and \cite{fraser-wheeler} previously used versions of the Convolution Stability Lemma of the type used in the present paper (rather than the type used in Kaufman's paper).

\begin{lem}[Convolution Stability Lemma]\label{convstab}
Let $\alpha > 0$. For $-1 < \beta < 1$, define 
\begin{equation*}
    N(\beta) = 99 \max(1, (1+\beta)^{-1}).   
\end{equation*}
Let $M_{k-1}$ and $M_k$ be real numbers satisfying 
\begin{equation*}
\min \left(M_{k-1}, M_{k-1}^{1+\beta}\right) \geq  100, 
\end{equation*}
\begin{equation}\label{Mcond1}
M_k \geq   4^{N(\beta)+1} M_{k-1}^{(2+\alpha)(N(\beta)+1)} (1 + M_{k-1}^{\beta}),
\end{equation}
\begin{equation}\label{Mcond2}
\log(M_k) \geq N(\beta) M_{k-1}^{3(2+\alpha)}.
\end{equation}
Let $G,H:\bZ \to \bC$ be such that $|G| \leq 1$, $G(0) = 1$, and $|H| \leq 2$. 
Assume $G$ and $H$ satisfy the following estimates: 
\begin{equation}\label{Fcond1}\tag{G1}
G(s)=0 \qquad \text{if} \ 1 \leq |s| < \min (M_k, M_k^{1+\beta}); 
\end{equation}
there is a constant $C > 0$ such that 
\begin{equation}\label{Fcond2}\tag{G2}
|G(s)| \leq C \log(|s|) M_k^{-1}(1+M_k^\beta) \qquad \text{if} \ |s| \geq \min (M_k, M_k^{1+\beta});  
\end{equation}
for each integer $N \geq 1$ there is a constant $C_{1,N} > 0$ such that 
\begin{equation}\label{Fcond3}\tag{G3}
|G(s)| \leq C_{1,N} \log(|s|) M_k^{-1+(2+\alpha)N}(1+M_k^\beta) |s|^{-N} \qquad \text{if} \ |s| \geq  M_k^{2+\alpha};
\end{equation}
for each integer $N \geq 1$ there is a constant $C_{2,N} > 0$ such that 
\begin{equation}\label{Gcond1}\tag{H}
|H(s)| \leq  C_{2,N} \log(|s|) M_{k-1}^{-1+(2+\alpha)(N+1)}(1+M_{k-1}^\beta) |s|^{-N} \qquad \text{if} \ |s| \geq 2M_{k-1}^{2+\alpha }. 
\end{equation}
Then the following estimates hold: 
\begin{equation}
\label{convlemres1}\tag{GH1}
|G \ast H(s) -H(s)| < C_{2,N(\beta)} M_k^{-96} \qquad \text{if} \ |s| \leq  \frac{1}{2} \min (M_k, M_k^{1+\beta}); 
\end{equation}
\begin{equation}
\label{convlemres2}\tag{GH2}
|G\ast H(s)| \leq (C+1)( C_{2,N(\beta)}+1) \log^2(|s|) M_k^{-1} (1+M_k^\beta) \qquad \text{if} \ |s| \geq  \frac{1}{2}\min (M_k, M_k^{1+\beta}); 
\end{equation}
\begin{align}
\label{convlemres3}\tag{GH3}
&|G\ast H(s)| \leq  (4^{N+1}C_{1,N+1}+C_{2,N+1})\log(|s|) M_k^{-1+(2+\alpha)(N+1)}(1+M_k^\beta) |s|^{-N} & \\ 
\notag 
&\hspace{0.6\textwidth}
\text{if} \ |s| \geq 2M_k^{2+\alpha}. &
\end{align}

\end{lem}
\begin{remark}
The indices on $M_{k-1}$ and $M_k$ in Lemma \ref{convstab} have no meaning at this point. We have used the notation because Lemma \ref{convstab} will be used iteratively later.  
In particular, note that the estimate \eqref{convlemres3} is of the same form as the estimate \eqref{Gcond1} with $k-1$ replaced by $k$. 
\end{remark}

We will prove the estimates on $G \ast H$ separately.

\begin{proof}[Proof of \eqref{convlemres1}]
Assume $|s| \leq \min (M_k, M_k^{1+\beta})/2$. We use \eqref{Fcond1}, that $G(0) = 1$, and that $|G| \leq 1$ to obtain 
\begin{equation*}
|G \ast H(s) - H(s)| = \sum_{|t| \geq \min (M_k, M_k^{1+\beta})} |G(t)H(s-t)|\leq \sum_{|t| \geq \min (M_k, M_k^{1+\beta})} |H(s-t)|.
\end{equation*}
Since $|t| \geq \min (M_k, M_k^{1+\beta})$ and $|s| \leq \min (M_k, M_k^{1+\beta})/2$, $|s-t| \geq |t|-|s| \geq \min (M_k, M_k^{1+\beta})/2$. By \eqref{Mcond1}, $\min (M_k, M_k^{1+\beta})/2 \geq 2M_{k-1}^{2+\alpha}$, and we can use \eqref{Gcond1} with $N=N(\beta)$. Therefore,
\begin{equation*}
\begin{split}
|G\ast H(s)-H(s)| &\leq \sum_{|u| \geq \min (M_k, M_k^{1+\beta})/2} |H(u)| \\
& \leq C_{2,N(\beta)}  M_{k-1}^{-1+(2+\alpha)(N(\beta)+1)}(1+M_{k-1}^\beta) \sum_{|u| \geq \min (M_k, M_k^{1+\beta})/2}\log(|u|)|u|^{-N(\beta)}\\
&\leq 2 \cdot C_{2,N(\beta)} \log(M_k)M_{k-1}^{-1+(2+\alpha)(N(\beta)+1)}(1+M_{k-1}^\beta) 4^{N(\beta)} M_k^{-98}   \\
&\leq C_{2,N(\beta)} M_k^{-96}.
\end{split} 
\end{equation*} 
In the last inequality, we used \eqref{Mcond1} and that $\log(|s|) \leq \log(M_k)\leq M_k$.
\end{proof}

\begin{proof}[Proof of \eqref{convlemres2}]
Assume $|s| \geq \min(M_k, M_k^{1+\beta})/2$. We write 
\begin{equation*}
|G \ast H| = \sum_{|t| < 2M_{k-1}^{2+\alpha}} |G(s-t)H(t)| + \sum_{|t| \geq 2M_{k-1}^{2+\alpha}} |G(s-t)H(t)|.
\end{equation*}
When $|t| < 2M_{k-1}^{2+\alpha}$, the assumptions $|s| \geq \min(M_k, M_k^{1+\beta})/2$ and \eqref{Mcond1} imply that $s-t \neq 0$ and $|s - t| \leq 2 |s|$.
Using \eqref{Fcond1}, \eqref{Fcond2} and the assumption that $|H| \leq 2$, we get
\begin{equation*}
\sum_{|t| < 2M_{k-1}^{2+\alpha}} |G(s-t)H(t)| \leq 2 C \log(2 |s|) M_k^{-1}(1+M_k^\beta) (4 M_{k-1}^{2+\alpha} + 1). 
\end{equation*}
By \eqref{Mcond2}, 
\begin{equation}\label{F*Gtsmall}
\sum_{|t| < 2M_{k-1}^{2+\alpha}} |G(s-t)H(t)| \leq C \log^2(|s|) M_k^{-1}(1+M_k^\beta).
\end{equation}
When $|t| \geq 2M_{k-1}^{2+\alpha}$, we consider the cases $t=s$ and $t \neq s$ separately. If $t=s$, we use $G(0) = 1$ and \eqref{Gcond1} with $N=N(\beta)$ to get
\begin{equation*}
\begin{split}
|G(0)H(s)| &\leq C_{2,N(\beta)}\log(|s|) M_{k-1}^{-1+(2+\alpha)(N(\beta)+1)}(1+M_{k-1}^\beta) |s|^{-N(\beta)}\\
&\leq C_{2,N(\beta)} 2^{N(\beta)}  \log(|s|) M_{k-1}^{-1+(2+\alpha)(N(\beta)+1)} (1+M_{k-1}^\beta) M_k^{-99}.
\end{split}
\end{equation*}
Using \eqref{Mcond1}, we obtain
\begin{equation}\label{F0Gs}
|G(0)H(s)| \leq C_{2,N(\beta)}\log(|s|) M_k^{-98}.
\end{equation}
If $t \neq s$, we use \eqref{Fcond1}, \eqref{Fcond2}, and \eqref{Gcond1} with $N=N(\beta)$ to get 
\begin{equation*}
\begin{split}    
\sum_{\substack{|t| \geq 2 M_{k-1}^{2+\alpha} \\ t \neq s}} |G(s-t)H(t)| \leq& C \cdot C_{2,N(\beta)}M_k^{-1}(1+M_k^\beta) M_{k-1}^{-1+(2+\alpha)(N(\beta)+1)} (1+M_{k-1}^\beta) \\
&\times \sum_{\substack{|t| \geq 2 M_{k-1}^{2+\alpha} \\ t \neq s}} \log(|s-t|) \log (|t|)|t|^{-N(\beta)}.
\end{split}
\end{equation*}
Since $|t| \geq 2M_{k-1}^{2+\alpha}$, $\log^2(|t|) \leq |t|$. Thus, we have
\begin{equation*}
    \begin{split}
        \sum_{\substack{|t| \geq 2 M_{k-1}^{2+\alpha} \\ t \neq s}} \log(|s-t|) \log(|t|) |t|^{-N(\beta)}&\leq \log(|s|) \sum_{|t| \geq 2M_{k-1}^{2+\alpha}} \log^2(|t|)|t|^{-N(\beta)}\\
        &\leq \log(|s|) \sum_{|t| \geq 2M_{k-1}^{2+\alpha}} |t|^{-N(\beta)+1}\\
        &\leq \log(|s|) M_{k-1}^{-(2+\alpha)(N(\beta)-2)}.
    \end{split}
\end{equation*}
Therefore,
\begin{equation*}
    \sum_{\substack{|t| \geq 2 M_{k-1}^{2+\alpha} \\ t \neq s}} |G(s-t)H(t)| \leq C \cdot C_{2,N(\beta)} \log(|s|)M_{k}^{-1}(1+M_k^\beta)M_{k-1}^{-1+3(2+\alpha)}(1+M_{k-1}^\beta).
\end{equation*}
Since $|s| \geq \min(M_k, M_k^{1 + \beta})/2$, by \eqref{Mcond2}, we obtain
\begin{equation}\label{FastG_tneqs}
\sum_{\substack{|t| \geq 2 M_k^{2+\alpha} \\ t \neq s}} |G(s-t)H(t)| \leq C \cdot C_{2,N(\beta)} \log^2(|s|) M_k^{-1}(1+M_k^\beta).
\end{equation}
Combining \eqref{F*Gtsmall}, \eqref{F0Gs} and \eqref{FastG_tneqs}, we get \eqref{convlemres2}.
\end{proof}

\begin{proof}[Proof of \eqref{convlemres3}]
We write 
\begin{equation*}
\begin{split}
|G\ast H(s)| = \sum_{|t| \leq |s|/2} |G(t)H(s-t)| +     \sum_{|t| \geq |s|/2 } |G(t)H(s-t)|. 
\end{split}
\end{equation*}
When $|t| \leq |s|/2$, we have $|t-s| \geq |s| -|t| \geq |s|/2$. 
Since $|G| \leq 1$, 
\begin{equation*}
\sum_{|t| \leq |s|/2} |G(t)H(s-t)| \leq \sum_{|t| \leq |s|/2} |H(s-t)| \leq \sum_{|u| \geq |s|/2} |H(u)|.
\end{equation*}
Since $|s| \geq 2M_k^{2+\alpha}$, the assumption \eqref{Mcond1} implies that $|s|/2 \geq 2M_{k-1}^{2+\alpha}$, so we can use \eqref{Gcond1}. By using \eqref{Gcond1} (with $N+1$ in place of $N$), we obtain 
\begin{equation*}
\begin{split}
\sum_{|t| \leq |s|/2} |G(t)H(s-t)| &\leq C_{2,N+1} M_{k-1}^{-1+(2+\alpha)(N+2)}(1+M_{k-1}^\beta)  \sum_{|u| \geq |s|/2} \log (|u|) |u|^{-(N+1)} \\
&\leq C_{2,N+1} \log(|s|) M_{k-1}^{-1+(2+\alpha)(N+2)}(1+M_{k-1}^\beta) 2 \cdot 4^N |s|^{-N}.
\end{split}
\end{equation*}
By using \eqref{Mcond1}, we get
\begin{equation*}
\begin{split}
2 \cdot 4^N M_{k-1}^{-1+(2+\alpha)(N+2)} (1+M_{k-1}^\beta) &\leq 4^{N+1} M_{k-1}^{(2+\alpha)(N+2)} \\
&\leq M_k^{(N+2)/100}\\
&\leq M_k^{-1+(2+\alpha)(N+1)}(1+M_{k}^\beta).
\end{split}
\end{equation*}
Thus, 
\begin{equation*}
\sum_{|t| \leq |s|/2} |G(t)H(s-t)| \leq C_{2,N+1}\log(|s|) M_k^{-1+(2+\alpha)(N+1)}(1+M_{k}^\beta) |s|^{-N}.
\end{equation*}
When $|t| \geq |s|/2 \geq M_k^{2+\alpha}$, 
we use \eqref{Fcond3} (with $N+1$ in place of $N$) and that $|H| \leq 2$. Thus 
\begin{equation*}
\begin{split}
\sum_{|t| \geq |s| /2 }|G(t)H(s-t)| &\leq 2 \sum_{|t| \geq |s|/2} |G(t)| \\
&\leq 2 C_{1, N+1} M_k^{-1+(2+\alpha)(N+1)}(1+M_{k}^\beta)\sum_{|t| \geq |s|/2}\log (|t|)|t|^{-(N+1)} \\
&\leq 2 C_{1, N+1} \log(|s|) M_k^{-1+(2+\alpha)(N+1)}(1+M_{k}^\beta)  2\cdot 4^N  |s|^{-N}.
\end{split}
\end{equation*}
Combining the two cases, we get \eqref{convlemres3}.
\end{proof}

\section{Construction and properties of the measure \texorpdfstring{$\mu$}{mu}}\label{sec_measure}

Let $\alpha > 0$ and $-1 < \beta < 1$ be given. In this section, we will construct the measure $\mu$, prove that it has support contained in $\bT \cap E(\alpha)$, and prove that it satisfies 
\begin{itemize}
\item  \eqref{MT_reg_mu} and \eqref{MT_FD_mu} in case $0 \leq \beta < 1$, or
\item  \eqref{MT_reg_mu_2} and \eqref{MT_FD_mu_2} in case $-1 < \beta < 0$.
\end{itemize}

\subsection{Preliminaries}

Let $F$ be a nonnegative Schwartz function supported on $\bT$ such that $\int F dx=1$ and $|F| \gtrsim 1$ if $|x| \leq 1/4$ and let $F_0$ be a nonnegative Schwartz function supported on $\bT$ such that $\int F_0 dx=1$ and $F_0$ is supported on $[-3/8,-1/8] \cup [1/8,3/8]$.

Since $F$ and $F_0$ are Schwartz functions, there are nondecreasing sequences of numbers $C_N, \widetilde{C}_N \geq 1$ (where $N=0,1,2\ldots$) such that 
\begin{align}\label{schwartz tail bound}
|\widehat{F}(s)| \leq C_N|s|^{-N} \quad \text{ and } \quad |\widehat{F}_0(s)| \leq \widetilde{C}_N|s|^{-N}
\end{align}
for each $|s| \geq 1$. Indeed, one may take 
\begin{equation*}
C_N = \max_{0 \leq i \leq N } \sup_{x \in \mathbb{T}}|\partial^i F | \qquad \text{and} \qquad \widetilde{C}_N = \max_{0 \leq i \leq N } \sup_{x \in \mathbb{T}}|\partial^i F_0 |.
\end{equation*}

Because $F$ is a Schwartz function, so is $\mathcal{F}(F^2)$. Since $F^2$ has a positive integral, $\mathcal{F}({F^2})(0) > 0$. We can therefore choose a constant $C_F$ such that $\mathcal{F}({F^2})(\xi) \gtrsim 1$ provided that $|\xi| \leq C_F$.

Let $\alpha > 0$ and $-1 < \beta < 1$. 
We fix a sequence $\{\epsilon_k\}_{k \in \bN}$ that converges to $0$. For each $k$, we choose $N_k$ such that $N_k \geq \max(N(\beta),k)$ and $N_k \epsilon_k \geq 2+\alpha$. 
 
Let $P_{M}$ denote the set of prime numbers strictly between $M$ and $2M$. Fix $D > 1$. We choose a rapidly increasing sequence of numbers $\{M_k\}_{k \in \bN \cup \{0\} }$ that satisfies the following:   
\begin{itemize}
\item $\min (M_0, M_0^{1+\beta}, M_0^\alpha) \geq 100$ and $\min(M_0, M_0^{1-\beta}) \geq \log(M_0)$.
\item For any $k$, 
\begin{equation*}
\frac{M_k}{{D}\log(M_k) } \leq |P_{M_k}| .
\end{equation*}
\item For any $k$,
\begin{equation}\label{Mcond4}
\log(M_k) \geq 
4^{N(\beta)+1} N(\beta) M_{k-1}^{(2+\alpha)(N(\beta)+1)}.
\end{equation}
\item For any $k$,
\begin{equation}\label{Mcond5}
\log\left(\frac{1}{2}\min (M_k, M_k^{1+\beta})\right) \geq \left(\frac{1}{1+\beta}+1 \right)^2(1+2{D}) \left( 1+ \widetilde{C}_{N_k+k}+ 8^{N_k+k+2}C_{N_k+k}D \ \right).
\end{equation}
\end{itemize}
Note that \eqref{Mcond4} implies \eqref{Mcond1} and \eqref{Mcond2} in Lemma \ref{convstab} since $M_k \geq \log(M_k)$.

Recall that 
$$
E(\alpha) = \cbr{x \in \RR : |x-r/q| \leq |q|^{-(2+\alpha)} \forinfmany (q,r) \in \ZZ \times \ZZ }. 
$$

For each $k \in \NN$, define  
\[
E(\alpha,k) = \cbr{ x \in \bR: |x-r/p| \leq \frac{M_k^{-(2+\alpha)}}{2} \text{ for some }  p \in P_{M_k}, r \in \ZZ, |r| \leq (p-1)/2  }.
\]
Note that $E(\alpha,k)$ is a subset of $\bT$. Note also that any point belonging to $E(\alpha, k)$ for infinitely many $k$ will belong to $E(\alpha)$.

For each odd $k \in \bN$, we choose a number $B_k \in \bN$ 
such that $M_k^{1+\beta} \leq B_k \leq 2M_k^{1+\beta}$. Additionally, if $\beta = 0$, we choose $B_k \in P_{M_k}$; if $\beta \neq 0$, we choose $B_k$ such that $B_k $ is coprime to every $p \in P_{M_k}$ for all sufficiently large $k$.

For each odd $k \in \bN$, define 
\begin{equation*}
C(\beta,k) = \cbr{ x\in \bR: \left|x- \frac{a}{B_k}\right| \leq \frac{M_k^{-(2+\alpha)}}{2}\ \text{ for some }   a \in \mathbb{Z}, |a| \leq (B_k-1)/2}.
\end{equation*}
Note that $C(\beta,k)$ is a subset of $\bT$. 

\begin{remark}
For $\beta = 0$, it is possible to make the same choice of $B_k$ as we do in the $\beta \neq 0$ case. However, the proof is somewhat simpler and more intuitive in the $\beta = 0$ case with the choice we have made. 
\end{remark}

\subsection{The single-scale functions \texorpdfstring{$g_k$}{gk}}

For even $k$, define 
\begin{equation*}
g_k(x) := \frac{M_k^{1+\alpha}}{|P_{M_k}|} \sum_{p \in P_{M_k}} \sum_{|r| \leq (p-1)/2} F(M_k^{1+\alpha}(px-r)).
\end{equation*}
For odd $k$, define 
\begin{equation*}
\begin{split}
g_k(x) := &\frac{M_k^{1+\alpha}}{|P_{M_k}|+M_k^\beta+1} \sum_{p \in P_{M_k}} \sum_{|r| \leq (p-1)/2} F(M_k^{1+\alpha}(px-r))\\
&+\frac{M_k^{1+\alpha}(1+M_k^{-\beta})}{|P_{M_k}|+M_k^\beta+1} \sum_{|a| \leq (B_k-1)/2} F(M_k^{1+\alpha-\beta}(B_kx-a)).
\end{split}
\end{equation*}
Note that, for every $k \in \NN$, 
$\int g_k dx =1$ 
and the support of $g_k$ is contained in $\bT$. 
If $k$ is even, the support of $g_k$ is contained in $E(\alpha, k)$. 
If $k$ is odd, the support of $g_k$ is contained in $E(\alpha, k) \cup C(\beta, k)$.

\begin{lem}\label{lem_gdecay}
For every $k \in \bN$  and $s \in \bZ$, the following statements hold.  
\begin{enumerate}[(1)]
\item $|\widehat{g}_k(s)| \leq 1$.
\item $\widehat{g}_k(s) =0 $ if $1 \leq |s| < \min(M_k, M_k^{1+\beta})$.
\item If $s \neq 0$,
\begin{equation}\label{g_kdecay_0}
    |\widehat{g}_k(s)| \leq 2 D \left(\frac{1}{1+\beta}+1 \right) \log(|s|)M_k^{-1}(1+M_k^\beta). 
\end{equation}
\item If $s \neq 0$, 
\begin{equation}\label{g_kdecay}
    |\widehat{g}_k(s)| \leq 2^{N+1} C_N D\left( \frac{1}{1+\beta}+1 \right) \log(|s|)M_k^{-1+(\alpha+2)N}(1+M_k^\beta) |s|^{-N}.
\end{equation}
\end{enumerate}
\end{lem}
\begin{proof}
Since $F$ and $g_k$ are supported on $\bT$, 
$\widehat{F}(s) = \mathcal{F}(F)(s)$ and $\widehat{g_k}(s) = \mathcal{F}(g_k)(s)$. 
We shall need the following calculation. 
Suppose $c>0$ and $b$ is an odd positive integer. In the following, the variable $a$ is an integer. For every $s \in \ZZ$, 
\begin{align*}
\int_{\RR} \sum_{|a| \leq (b-1)/2} F(c(bx-a)) e^{-2\pi i x s} dx
=
\dfrac{1}{cb} \widehat{F}\rbr{\dfrac{s}{cb}} e^{\pi i (b-1) s / b} \sum_{a=0}^{b-1} e^{-2\pi i a s / b}
\end{align*}
The sum on the right equals $b$ if $b$ divides $s$ and equals $0$ otherwise. 
We now apply this calculation to compute $\widehat{g}_k$ and prove (1), (2), (3), and (4). 

First assume $k$ is even. We have 
\[\widehat{g}_k(s) = \frac{1}{|P_{M_k}|} \sum_{p \in P_{M_K}} e^{\pi i (p-1) s / p} \widehat{F} \left( \frac{s}{p M_k^{1 + \alpha}} \right) \mathbf{1}_{p \bZ}(s),\]
where $\mathbf{1}_{p\bZ}(s) = 1$ if $p \mid s$ and $0$ otherwise. Hence (1) and (2) follow easily.
It remains to establish (3) and (4). By \eqref{schwartz tail bound}, we have the estimate
\begin{equation*}
\left| \widehat{F} \left(\frac{s}{pM_k^{1+\alpha}} \right) \right| \leq \min (1, 2^N C_N M_k^{(2+\alpha)N}|s|^{-N} ).
\end{equation*}
Since we have already established (2) for $1 \leq |s| < M_k$, it is enough to consider $|s| \geq M_k$. Since there are at most $\frac{\log |s|}{\log M_k}$ values of $p \in P_{M_k}$ such that $p \mid s$, we have
\begin{align*}
|\hat g_k(s)| & \leq D M_k^{-1} \bigg(\log (|s|) + \log (M_k) \bigg) \min(1, 2^N C_N M_k^{(2 + \alpha)N} |s|^{-N}) \\
& \leq 2D M_k^{-1}\log(|s|) \min (1, 2^N C_N M_k^{N(2+\alpha)}|s|^{-N} ), 
\end{align*}
which immediately gives (3) and (4). 

Now consider the case where $k$ is odd. We have
\begin{equation}\label{g_khat}
\begin{split}
\widehat{g}_k(s) = 
&\frac{1}{|P_{M_k}|+M_k^\beta+1}  \sum_{p \in P_{M_k}} e^{\pi i (p-1) s / p} \widehat{F} \left( \frac{s}{pM_k^{1+\alpha}} \right)\mathbf{1}_{p\bZ}(s) 
\\
&+ 
\frac{e^{\pi i (B_k-1) s / B_k} (1+M_k^\beta) }{|P_{M_k}|+M_k^\beta+1} e^{\pi i (B_k-1)/B_k}
\widehat{F}\left( \frac{s}{B_kM_k^{\alpha-\beta+1}} \right) \mathbf{1}_{B_k\bZ}(s) , 
\end{split}
\end{equation}
where $\mathbf{1}_{p\bZ} (s)$ is defined as before and $\mathbf{1}_{B_k\bZ} (s) = 1$ if $B_k |s$ and $\mathbf{1}_{B_k\bZ}(s) =0$ otherwise. Thus, (1) and (2) easily follow. 
By \eqref{schwartz tail bound}, 
\begin{equation*}
\abs{ \widehat{F} \left(\frac{s}{pM_k^{1+\alpha}} \right) } \leq \min (1, 2^N C_N M_k^{(2+\alpha)N}|s|^{-N} )
\end{equation*}
and
\begin{equation*}
    \abs{ \widehat{F} \left(\frac{s}{B_k M_k^{\alpha-\beta+1}} \right) } \leq \min (1, 2^N C_N M_k^{(2+\alpha)N}|s|^{-N} ).
\end{equation*}
If $\beta\geq 0$, $\widehat{g}_k(s) = 0$ if $1 \leq s < M_k$ by (2). Thus, it suffices to consider $|s| \geq M_k$. Using that the number of $p \in P_{M_k}$ such that $p|s$ is bounded by $\frac{\log|s|}{\log M_k}$, we obtain
\begin{equation*}
\begin{split}
    |\widehat{g}_k(s)| &\leq DM_k^{-1} \left( \log(|s| ) +\log(M_k)(1+M_k^\beta) \right) \min (1, 2^N C_N M_k^{(2+\alpha)N}|s|^{-N} )\\
    &\leq 2D M_k^{-1}\log(|s|) (1+M_k^\beta)\min (1, 2^N C_N M_k^{N(2+\alpha)}|s|^{-N} ).
\end{split}
\end{equation*}
If $-1< \beta<0$, it suffices to consider $|s| \geq M_k^{1+\beta}$. Similarly, we obtain that 
\begin{equation*}
    |\widehat{g}_k(s)| \leq D M_k^{-1}\log(|s|) \left(1+\frac{1}{1+\beta}(1+M_k^\beta)\right)\min (1, 2^N C_N M_k^{(2+\alpha)N}|s|^{-N} ).
\end{equation*}
Therefore, we get (3) and (4). 
\end{proof}

\subsection{The measures \texorpdfstring{$\mu_k$}{muk}}

We define a sequence of functions
\begin{equation*}
\mu_k = g_k \cdots g_1 F_0   
\end{equation*}
and we identify $\mu_k$ with the measure $\mu_k dx$.

\begin{lem}\label{mu_induction}
The following statements hold for all $k \in \bN $ and $s \in \bZ$: 
\begin{enumerate}[(1)]
    \item We have 
\begin{equation}\label{convgMres0}
    |\widehat{\mu}_k(s) | \leq 2.
\end{equation}
\item If $|s| \leq \min(M_k,M_k^{1+\beta})/2$,
    \begin{equation}\label{convgMres1}
|\widehat{\mu }_{k} (s) - \widehat{\mu }_{k-1}(s)| \leq \log (M_k)M_k^{-96}.
\end{equation}
\item If $|s| \geq \min(M_k,M_k^{1+\beta})/2$,
\begin{equation}\label{convgMres2}
|\widehat{\mu }_{k}(s) | \leq \log^3(|s|) M_k^{-1}(1+M_k^\beta).
\end{equation}
\item If $|s| \geq 2M_k^{2+\alpha}$, for any $N \in \bN$,
\begin{equation}\label{convgMres3}
|\widehat{\mu }_{k}(s) | \leq \left(\widetilde{C}_{N+k}+ 2 D \left( \frac{1}{1+\beta}+1 \right)\sum_{j=1}^{k}8^{N+j}C_{N+j}\right) \log(|s|) M_{k}^{-1+(2+\alpha)(N+1)}(1+M_k^\beta) |s|^{-N}.
\end{equation} 
\end{enumerate}
\end{lem}
\begin{proof}
First, we prove by induction that the following hold for all $k \in \bN$:
\begin{enumerate}[(i)]
    \item If $|s| \leq  \min(M_k,M_k^{1+\beta})/2$,
    \begin{equation*}
|\widehat{\mu }_{k-1} \ast \widehat{g}_k(s) - \widehat{\mu }_{k-1}(s)| \leq \left(\widetilde{C}_{N(\beta)+k-1}+2D\left(\frac{1}{1+\beta}+1 \right) \sum_{j=1}^{k} 8^{N(\beta)+j-1} C_{N(\beta)+j-1}\right)M_k^{-96}.
\end{equation*}
\item If $|s| \geq  \min(M_k,M_k^{1+\beta})/2$,
\begin{equation*}
\begin{split}
|\widehat{\mu }_{k-1} \ast \widehat{g}_k(s) | \leq &\left(\frac{1}{1+\beta}+1 \right)^2 \left(1+2{D}\right) \left(1+\widetilde{C}_{N(\beta)+k-1}+ 2D \sum_{j=1}^{k} 8^{N(\beta)+j-1} C_{N(\beta)+j-1}\right)\\
&\times \log^2(|s|) M_k^{-1}(1+M_k^\beta).
\end{split}
\end{equation*}
\item If $|s| \geq 2M_k^{2+\alpha}$, for any $N \in \bN$,
\begin{equation*}
\begin{split}
|\widehat{\mu }_{k-1} \ast \widehat{g}_k(s) | \leq&\left(\widetilde{C}_{N+k}+ 2D \left( \frac{1}{1+\beta}+1 \right)\sum_{j=1}^{k}8^{N+j}C_{N+j}\right)\\
&\times \log(|s|) M_{k}^{-1+(2+\alpha)(N+1)}(1+M_k^\beta) |s|^{-N}.
\end{split}
\end{equation*}
\end{enumerate} 

    When $k=1$, $\widehat{\mu}_0 = \widehat{F}_0$ where $F_0$ is a Schwartz function. 
    Let $G = \widehat{g}_1$ with $C=2D((1+\beta)^{-1}+1)$ and $C_{1,N} = 2^{N+1}C_ND((1+\beta)^{-1}+1)$ and 
    let $H =\widehat{\mu}_0$ with $C_{2,N} = \widetilde{C}_N$. By Lemma \ref{lem_gdecay}, $\widehat{g}_1$ satisfies the assumptions on $G$ in Lemma \ref{convstab}. 
    Because $|\widehat{\mu}_0(s) | \leq 1$ and because of \eqref{schwartz tail bound}, $\widehat{\mu}_0$ satisfies the assumptions on $H$ in Lemma \ref{convstab}. Thus, we have (i), (ii) and (iii) when $k=1$.
    
    Now we assume (i), (ii), (iii) when $k=1, 2, \cdots, i-1$ and prove the same statements with $i-1$ replaced by $i$. Let $G = \widehat{g}_i$ with $C= 2D((1+\beta)^{-1}+1)$ and $C_{1,N}=2^{N+1} C_{N}D ((1+\beta)^{-1}+1)$ as before, and let $H=\widehat{\mu }_{i-1} $ with $C_{2,N} = \widetilde{C}_{N+i-1}+ 2 D((1+\beta)^{-1}+1)\sum_{j=1}^{i-1}8^{N+j}C_{N+j}$. By Lemma \ref{lem_gdecay}, $\widehat{g}_i$ satisfies the assumptions on $G$ in Lemma \ref{convstab}. 

    Note that $|\widehat{\mu}_{i-1}(s)| \leq |\widehat{\mu}_{i-1}(0)|$ since $\mu_{i-1}$ is nonnegative. The condition \eqref{Mcond4} implies that $M_k \geq \log(M_k) \geq 2^k$ and \eqref{Mcond5} implies that
    \begin{equation}\label{sumcons_1}
    \begin{split}
    \left(\widetilde{C}_{N(\beta)+k}+ {2D}\left(\frac{1}{1+\beta}+1 \right)\sum_{j=1}^{k} 8^{N(\beta)+j} C_{N(\beta)+j}\right) &\leq \log\left(\frac{1}{2}\min (M_k, M_k^{1+\beta})\right)\\
    &\leq \log(M_k).
    \end{split}
    \end{equation}
    Thus, we use (i) for $k=1, 2, \cdots, i-1$ and obtain
    \begin{equation}\label{mu_i-1est}
        |\widehat{\mu}_{i-1} (0)| \leq |\widehat{\mu}_0(0)| + \sum_{j=1}^{i-1} \log (M_i) M_i^{-96} \leq 2. 
    \end{equation}
    Also, applying (iii) with $k=i-1$ shows that $\hat \mu_{i-1}(s)$ satisfies the assumption \eqref{Gcond1} with $k=i$. Therefore, $\widehat{\mu}_{i-1}$ satisfies the assumptions on $H$ in Lemma \ref{convstab}. Thus, (i), (ii) and (iii) with $k=i$ hold by Lemma \ref{convstab}. By induction, (i),(ii) and (iii) hold for any $k \in \bN$.

    We already proved (4) by (iii). Combining (i) and \eqref{sumcons_1}, we get (2). Similarly, (3) follows from (ii) and \eqref{Mcond5}. Since \eqref{mu_i-1est} holds for any $i \in \bN$, we get (1).
\end{proof}

\subsection{Definition and support of \texorpdfstring{$\mu$}{mu}}

We now introduce the measure $\mu$. 
Since $\mu_0(\RR) = \widehat{\mu_0}(0) = \int F_0 dx = 1$ and since $M_k$ grows rapidly, \eqref{convgMres1} implies $\mu_k(\RR) = \widehat{\mu_k}(0) \in (1/2,2)$ for all $k$. Of course, each $\mu_k$ is supported on the compact set $\TT$. 
Therefore Prohorov's theorem (see \cite[vol.2, p.202]{bogachev}) implies the sequence $(\mu_k)_{k=1}^{\infty}$ has a subsequence $(\mu_{k_j})_{j=1}^{\infty}$ 
which converges weakly to a non-zero finite measure $\mu$. In fact, the Fourier decay estimates on $\mu_k$ and an argument involving the L\'evy continuity theorem allow us to prove the full sequence $(\mu_k)_{k=1}^{\infty}$ converges to $\mu$ (see e.g., \cite{hambrook-tams}), but we shall not need that fact here. By multiplying $\mu$ by an appropriate constant, it becomes a probability measure. 

From the definition of $\mu_k$, it is clear the support of $\mu$ is contained in $\TT \cap \bigcap_{k \in \NN} \supp(g_k)$. 
But recall that the support of $g_k$ is contained in $E(\alpha, k)$ if $k$ is even and is contained in $E(\alpha, k) \cup C(\beta, k)$ if $k$ is odd. 
Recall also that any point that belongs to $E(\alpha, k)$ for infinitely many $k$ must belong to $E(\alpha)$. Therefore 
$$
\bigcap_{k \in \NN} \supp(g_k) 
\subseteq 
\bigcap_{k \text{ even}} E(\alpha, k) \cap \bigcap_{k \text{ odd}} \rbr{E(\alpha, k) \cup C(\beta, k)}
\subseteq 
\bigcap_{k \text{ even}} E(\alpha, k) \subseteq E(\alpha). 
$$
This shows that $\supp(\mu) \subseteq \TT \cap E(\alpha)$. 

\subsection{Fourier decay of \texorpdfstring{$\mu$}{mu}}

We now show that $\mu$ satisfies \eqref{MT_FD_mu} if $\beta \geq 0$ and \eqref{MT_FD_mu_2} if $\beta < 0$. 

\begin{prop}\label{lem_FD_mu}
If $0 \leq \beta <1$, the measure $\mu$ satisfies
\begin{equation}\label{decaymu}
|\widehat{\mu}(s) | \lesssim_\epsilon \log^3(|s|) |s|^{-\frac{1-\beta}{2+\alpha+\epsilon}}, \qquad \text{for every $s \in \ZZ$}.
\end{equation}
If $-1 < \beta <0$, the measure $\mu$ satisfies
\begin{equation}\label{decaymu_2}
|\widehat{\mu}(s) | \lesssim_\epsilon \log^3(|s|)|s|^{-\frac{1}{2+\alpha+\epsilon}}, \qquad \text{for every $s \in \ZZ$}.
\end{equation}
\end{prop}
\begin{proof}
First, we claim that for all $k \in \bN$,
\begin{equation}\label{j-1muhat_L}
|\widehat{\mu}_{k}(s) | \leq \log^3(|s|)(|s|^{-\frac{1}{2+\alpha+\epsilon_k}}+|s|^{-\frac{1-\beta}{2+\alpha+\epsilon_k}}) \qquad \text{if} \ |s| \geq   \min(M_{k}, M_k^{1+\beta})/2.
\end{equation}

If $\min(M_k, M_k^{1+\beta})/2 \leq |s| \leq 2 M_k^{2+\alpha+\epsilon_{k}}$, it follows from \eqref{convgMres2} that
\begin{equation*}
\begin{split}
|\widehat{\mu}_k(s) | &\leq  \log^3(|s|)M_k^{-1}(1+M_k^\beta)\\
&\lesssim \log^3 (|s|)(|s|^{-\frac{1}{2+\alpha+\epsilon_k}}+|s|^{-\frac{1-\beta}{2+\alpha+\epsilon_k}}).
\end{split}
\end{equation*}
If $|s| \geq 2M_k^{2+\alpha+\epsilon_k}$, \eqref{Mcond5} and \eqref{convgMres3}  imply that 
\begin{equation*}
\begin{split}
|\widehat{\mu_k}(s)| &\leq \log(M_k) \log(|s|) M_k^{-1+(2+\alpha)(N_k+1)}(1+M_k^\beta)|s|^{-N_k}\\
&\leq \log^2(|s|) |s|^{\frac{-1+(2+\alpha)-\epsilon_k N_k}{2+\alpha+\epsilon_k}} (1+ |s|^{\frac{\beta}{2+\alpha+\epsilon_k}}).
\end{split}
\end{equation*}
Since $\epsilon_k N_k \geq 2+\alpha$, we obtain
\begin{equation*}
        |\widehat{\mu_k}(s)| \leq \log^2(|s|) (|s|^{-\frac{1}{2+\alpha+\epsilon_k}}+|s|^{-\frac{1-\beta}{2+\alpha+\epsilon_k}}).
\end{equation*}

We have established \eqref{j-1muhat_L}.

For a given $\epsilon>0$, we choose $k$ such that $\epsilon_k \leq \epsilon$. For $\ell >k$, \eqref{j-1muhat_L} implies that
\begin{equation}\label{mulhat_L}
|\widehat{\mu}_{l}(s) | \lesssim \log^3 (|s|)(|s|^{-\frac{1}{2+\alpha+\epsilon}}+|s|^{-\frac{1-\beta}{2+\alpha+\epsilon}}) \qquad \text{if} \ |s| \geq \min(M_l, M_l^{1+\beta})/2 .
\end{equation}
If $\min(M_{i-1}, M_{i-1}^{1+\beta})/2  \leq |s| \leq \min(M_i, M_i^{1+\beta})/2 $ for some $k+1 \leq i\leq l$, we use \eqref{convgMres1} repeatedly and obtain that
\begin{equation}\label{mulhat_S}
\begin{split}
    |\widehat{\mu}_l(s)| &\leq |\widehat{\mu}_{i-1}(s)| + \sum_{j=i}^l \log(M_j) M_j^{-96}\\
    &\leq \log^2(|s|)(|s|^{-\frac{1}{2+\alpha+\epsilon}}+|s|^{-\frac{1-\beta}{2+\alpha+\epsilon}}) +|s|^{-1}.
\end{split}
\end{equation}
Combining \eqref{mulhat_L} and \eqref{mulhat_S}, we get
\begin{equation*}
|\widehat{\mu}_{\ell}(s) | \lesssim  \log^3 (|s|)(|s|^{-\frac{1}{2+\alpha+\epsilon}}+|s|^{-\frac{1-\beta}{2+\alpha+\epsilon}}) \qquad \text{if} \ |s| \geq \min( M_{k}, M_k^{1+\beta})/2
\end{equation*}
for any $\ell \geq k$. Letting $\ell \rightarrow \infty$, we obtain \eqref{decaymu} and \eqref{decaymu_2}. Note that the implicit constant in \eqref{decaymu} and \eqref{decaymu_2} depends on $\epsilon $, since our choice of $k$ depends on $\epsilon$.
\end{proof}

Since $\supp(\mu) \subseteq \TT$, 
$\widehat{\mu}(\xi) = \mathcal{F}(\mu)(\xi)$ for integer $\xi$.
So the preceding proposition shows that $\mu$ satisfies 
\eqref{MT_FD_mu} and \eqref{MT_FD_mu_2} when $\xi \in \ZZ$. 
In fact, by a standard argument (see, e.g., \cite[p.252-253]{kahane-book} or \cite[p.69]{wolff-book}) the same estimates extend to all $\xi \in \RR$. 
Thus we have proved $\mu$ satisfies \eqref{MT_FD_mu} if $\beta \geq 0$ and \eqref{MT_FD_mu_2} if $\beta < 0$. 

\subsection{Regularity of \texorpdfstring{$\mu$}{mu}}

We now show that $\mu$ satisfies \eqref{MT_reg_mu} if $\beta \geq 0$ and \eqref{MT_reg_mu} if $\beta \leq 0$. We start with some definitions and lemmas. 

For $ l >k $, define 
$$
\mu_{l,k} :=  
g_l \cdots g_{k+1} \mu_{k-1} 
=
g_l \cdots g_{k+1} g_{k-1} \cdots g_1 \mu_0
$$
and 
$\mu_{k,k} := \mu_{k-1}$. 
In other words, $\mu_{l,k}$ is defined by omitting $g_k$ from the product that defines $\mu_l$. 
By re-indexing, the proof of equation \eqref{convgMres0} and \eqref{convgMres1} in Lemma \ref{mu_induction} yields the following lemma, which will be used in the proof of regularity of $\mu$ and the failure of the extension estimate (R). 

\begin{lem}
\begin{equation}\label{mukleq2}
\left| \widehat{\mu}_{l,k} (s)  \right| \leq 2 \qquad \text{for every $s \in \ZZ$} \end{equation}
and
\begin{equation}\label{mukleq}
\left| \widehat{\mu}_{l,k} (s) - \widehat{\mu}_{l-1,k}(s) \right| \leq \log(M_l)M_l^{-96} \qquad \text{if}\  |s| \leq \min( M_{l}, M_l^{1+\beta})/2. \end{equation}
\end{lem}

To prove the regularity of the measure $\mu$, we define $k$-interval.
\begin{defn} 
If $k$ is odd, a \textit{$k$-interval} is a connected component of $E(\alpha, k) \cup C(\beta, k)$. 
If $k$ is even, a \textit{$k$-interval} is a connected component of $E(\alpha, k)$.
\end{defn}
Note (i) the intervals of $E(\alpha,k)$ and of $C(\beta, k)$ are each of length $M_k^{-(2+\alpha)}$,  (ii) the intervals of $E(\alpha,k)$ which do not contain  the origin are $M_k^{-2}$-separated, and (iii) the intervals of $C(\beta, k)$ are $M_k^{-(1+\beta)}$-separated. Therefore each $k$-interval which does not contain the origin is a union of at most one interval of $E(\alpha,k)$ and at most one interval of $C(\beta, k)$. Also, observe that if $\beta = 0$, each interval in $C(\beta, k)$ is in $E(\alpha, k)$, so a $k$-interval is just an interval of $E(\alpha, k)$ in this special case.

\begin{lem}\label{mukFIhat}
Recall that $F$ is a nonnegative Schwartz function supported on $\bT$. Let $I$ be a $k$-interval centered at $x_I $ such that $|x_I| \geq 10M_k^{-(2+\alpha)}$ and define 
\begin{equation*}
F_I(x) = F\left( \frac{1}{2|I|} (x-x_I) \right).
\end{equation*}
Then, for any $k \in \bN$,
\begin{equation*}
|\widehat{g_kF_I}(s)| \lesssim_N \log(M_k)M_k^{-2}(1+M_k^{-\beta})\left( 1+\frac{|s|}{M_k^{2+\alpha}} \right)^{-N}.
\end{equation*}
\end{lem}
\begin{proof}
For a $k$-interval $I$, let $2I$ be an interval of length $2|I|$ with the same center as $I$. The function $F_I$ is supported on an interval of length $2|I| \sim M_{k}^{-(2+\alpha)}$ and its support does not intersect a $k$-interval which contains the origin. Therefore, we have
\begin{equation*}
\sum_{p \in P_{M_k}} \sum_{|r| \leq (p-1)/2} F(M_k^{1+\alpha}(px-r))F_I(x) = \sum_{(p,r) \in P_{M_k}(I)} F(M_k^{1+\alpha}(px-r))F_I(x)
\end{equation*}
for some set $P_{M_k}(I)$ which contains $\lesssim 1$ pairs $(p,r)$. 
Likewise, for odd $k$, we have 
\begin{equation*}
\sum_{|a| \leq (B_k-1)/2} F(M_k^{\alpha-\beta+1}(B_kx-a))F_I(x) = \sum_{a \in B_k(I)}F(M_k^{\alpha-\beta+1}(B_kx-a))F_I(x) 
\end{equation*}
for some set $B_k(I)$ of size $\lesssim 1$.
Therefore, we have for even $k$ that
\[g_k F_I(x) = \frac{M_k^{1 + \alpha}}{|P_{M_k}|} \sum_{(p, r) \in P_{M_k}(I)} F(M_k^{1 + \alpha} (px - r)) F_I(x),\]
and we have for odd $k$ that
\begin{equation*}
\begin{split}
g_kF_I(x) = &\frac{M_k^{1+\alpha}}{|P_{M_k}|+M_k^\beta+1}\sum_{(p,r) \in P_{M_k}(I) } F(M_k^{1+\alpha}(px-r))F_I(x) \\
&+\frac{M_k^{1+\alpha}(1+M_k^{-\beta})}{|P_{M_k}|+M_k^\beta+1}\sum_{a \in B_k(I)} F(M_k^{\alpha-\beta+1}(B_kx-a))F_I(x).
\end{split}
\end{equation*}
The functions 
\begin{equation*}
F(M_k^{1+\alpha}(px-r))F_I(x) \qquad \text{and} \qquad F(M_k^{\alpha-\beta+1}(B_kx-a))F_I(x)
\end{equation*}
are Schwartz functions whose supports are contained in subintervals of $\bT$ of length $\sim M_k^{-(2+\alpha)}$. Thus, their Fourier coefficients at $s$ are bounded above by a constant multiple of
\begin{equation*}
M_{k}^{-(2+\alpha)} \left(1+ \frac{|s|}{M_{k}^{2+\alpha}} \right)^{-N}
\end{equation*}
where the implicit constant depends on $N$. We then obtain for even $k$ that
\[|\widehat{g_k F_I}(s)| \lesssim_N \frac{\log M_k}{M_k^2} \sum_{(p,r) \in P_{M_k}(I)} \left(1 + \frac{|s|}{M_k^{2 + \alpha}} \right)^{-N}, \]
and for odd $k$ that
\begin{equation*}
|\widehat{g_kF_I}(s)| \lesssim_N \frac{\log(M_k)}{M_k^2} \left( \sum_{(p,r) \in P_{M_k}(I)} \left(1+ \frac{|s|}{M_k^{2+\alpha}} \right)^{-N}+ (1+M_k^{-\beta})\sum_{a \in B_k(I)} \left(1+ \frac{|s|}{M_k^{2+\alpha}} \right)^{-N} \right).
\end{equation*}
Since the sizes of $P_{M_k}(I)$ and $B_k(I)$ are $\lesssim 1$, we establish the inequality.
\end{proof}
We will consider two kinds of $k$-intervals. Let $\mathcal{J}_2$ be the set of $k$-intervals $I$ such that $2I$ intersects $C(\beta,k)$ and let $\mathcal{J}_1$ be the remaining $k$-intervals. Of course, any $k$-interval in $\mathcal{J}_1$ will be an interval from $E(\alpha, k)$. Notice that if $k$ is even, then $\mathcal{J}_2$ is empty.

\begin{lem}\label{muI_kint}
    If $I$ is a $k$-interval in $\mathcal{J}_1$,
    \begin{equation}\label{muI_kint_1}
        \mu(I) \lesssim \log^3(M_k) M_k^{-2},
    \end{equation}
    and
    \begin{equation}\label{muI_kint_2}
        \mu(I) \lesssim \log^3(M_{k-1})M_{k-1}^\alpha (2 + M_{k-1}^{-\beta}) \log(M_k)M_k^{-2}.
    \end{equation}
    If $I$ is a $k$-interval in $\mathcal{J}_2$,
    \begin{equation}\label{muI_kint_3}
        \mu(I) \lesssim \log^3(M_k) M_k^{-2} (2 +M_k^{-\beta}).
    \end{equation}
\end{lem}
\begin{proof} We consider $\mu_k$ first.
\paragraph{\textbf{Estimate of $\mu_k$ for $k$-intervals in $\mathcal{J}_1$}} Let $I \in \mathcal{J}_1$. We will find an upper bound on $\mu_k(2I)$. Observe that $\mu_k(2I)$ is given by the integral
\[\mu_k(2I)  = \int F_0(x) g_1(x) \cdots g_k(x) \mathbf{1}_{2I}(x) \, dx.\]
The definition of $g_k$ and the fact that $F_0$ is supported away from the origin guarantee that on the support of $F_0$ for even values of $k$ we have the estimate
\[g_k(x)  \leq \frac{C_0 M_k^{1 + \alpha}}{|P_{M_k}|}.\]
For odd values of $k$, $2I$ does not intersect $C(\beta,k)$ since $I \in \mathcal{J}_1$. Thus, the definition of $g_k$ and the support of $F_0$ guarantee that we have the estimate
\begin{equation*}
g_k(x) \leq \frac{C_0 M_k^{1 + \alpha}}{|P_{M_k}| + M_k^{\beta} + 1} \qquad \text{for} \ x \in 2I.
\end{equation*}
In either case, we have the estimate
\begin{equation}\label{g_konI}
g_k(x) \leq C_0 D \log(M_k) M_k^{\alpha} \qquad \text{for} \ x \in 2I.
\end{equation}
From \eqref{Mcond5}, we get $C_0 D \leq \log(M_k)$. Therefore,
\begin{equation*}
g_k(x) \leq \log^2(M_k) M_k^{\alpha} \leq M_k^{1+\alpha},
\end{equation*}
and for $1 \leq j \leq k-1$, we have the weaker estimate
\begin{equation*}
    g_j(x) \leq C_0 D \log (M_j) M_j^{\alpha} (2 + M_j^{-\beta}) \leq \log^2(M_j) M_j^{\alpha} (2 + M_j^{-\beta}).
\end{equation*}
By \eqref{Mcond4}, $\log(M_k) M_k^{1+\alpha} \leq \log(M_{k+1})$. For $x \in 2I$, we get
\begin{equation*}
\begin{split}
g_1 \cdots g_k(x) &\leq \left(\prod_{j=1}^{k-1} \log^2(M_j) M_j^{\alpha} (2 + M_j^{-\beta}) \right) \log^2(M_k) M_k^\alpha\\
&\leq \log(M_2) \left(\prod_{j=2}^{k-1} \log^2(M_j) M_j^{\alpha} (2 + M_j^{-\beta}) \right) \log^2(M_k) M_k^\alpha\\
&\leq \log(M_3) \left(\prod_{j=3}^k \log^2(M_j) M_j^{\alpha} (2 + M_j^{-\beta})\right) \log^2(M_k) M_k^\alpha\\
&\qquad \vdots\\
&\leq \log^3(M_{k})M_k^\alpha.
\end{split}
\end{equation*}
Thus, we obtain
\[\mu_k(2I) \lesssim \log^3(M_k) M_k^\alpha  \int \mathbf{1}_{2I}(x) \, dx.\]
Since $k$-intervals have Lebesgue measure $\sim M_k^{-(2 + \alpha)}$, we have established the bound
\begin{equation}\label{mu_k_kint_epsilon}
\mu_k(2I) \lesssim \log^3(M_k) M_k^{-2}.
\end{equation}
Also, by truncating the above argument one step earlier and using \eqref{g_konI}, we have that
\begin{equation}\label{mu_k_kint}
\mu_k(2I) \lesssim \int (g_1 \cdots g_{k-1})(x) g_k(x) \mathbf{1}_{2I}(x) dx  \lesssim \log^3(M_{k-1}) M_{k-1}^{\alpha} (2 + M_{k-1}^{-\beta}) \log(M_{k})M_{k}^{-2}.
\end{equation}

\paragraph{\textbf{Estimate of $\mu_k$ for $k$-intervals in $\mathcal{J}_2$}} Now, let $I \in \mathcal{J}_2$. We will find an upper bound on $\mu_k(2I)$. Since $\mathcal{J}_2$ is nonempty, $k$ must be odd. Instead of \eqref{g_konI}, we have
\begin{equation*}
g_k(x) \leq  C_0 D \log(M_k) M_k^{\alpha}(2+M_k^{-\beta}) \qquad \text{for} \ x \in 2I.
\end{equation*}
Similarly, we get
\begin{equation*}
    g_1\cdots g_k(x) \leq \log^3(M_k)M_k^\alpha (2+M_k^{-\beta}) \qquad \text{for} \ x \in 2I. 
\end{equation*}
Therefore,
\begin{equation}\label{mu_k_kint_2}
    \mu_k(2I) \lesssim \log^3(M_k)M_k^{-2}(2+M_k^{-\beta}).
\end{equation}

\paragraph{\textbf{Estimate of $\mu$ for $k$-intervals}} 
Let $I$ be a $k$-interval. 
We seek to estimate $\mu_l(I)$ for $l > k$. Let $F_I$ be the nonnegative Schwartz function defined in Lemma \ref{mukFIhat} and write 
\begin{equation*}
\mu_l(F_I) := \int F_I \mu_l dx. 
\end{equation*}
We apply Plancherel's theorem to write 
\begin{equation*}
\mu_l(F_I) = \sum_{s \in \bZ} \widehat{g_kF_I}(s) (g_{l} \cdots g_{k+1} \mu_{k-1})^\wedge(s) \qquad \text{and} \qquad \mu_k(F_I) = \sum_{s\in \bZ} \widehat{g_kF_I}(s) \widehat{\mu}_{k-1}(s).
\end{equation*}
Let us consider
\begin{equation*}
\mu_l(F_I) - \mu_k(F_I) = \sum_{s \in \bZ} \widehat{g_kF_I}(s) ((g_l \cdots g_{k+1}\mu_{k-1})^\wedge (s) - \widehat{\mu}_{k-1}(s)) .
\end{equation*}
If $|s| \geq M_k^{10(2+\alpha)}$, we use the the decay of $\widehat{g_kF_I}$. By \eqref{convgMres0} and \eqref{mukleq2} we have 
\begin{equation*}
|(g_l \cdots g_{k+1}\mu_{k-1})^\wedge (s)- \widehat{\mu}_{k-1}(s)| \lesssim 1.
\end{equation*}
By Lemma \ref{mukFIhat} with $N=10$, we obtain that
\begin{equation*}
\begin{split}
\sum_{|s| \geq M_k^{10(2+\alpha)}} &|\widehat{g_k F_I}(s) ((g_l \cdots g_{k+1} \mu_{k-1})^\wedge (s) -\widehat{\mu}_{k-1}(s))|\\
&\lesssim \log(M_k)M_k^{-2}(1+M_k^{-\beta})\sum_{|s| \geq M_k^{10(2+\alpha)}}   \left( \frac{M_k^{2+\alpha}}{|s|} \right)^{10}\lesssim M_k^{-80(2+\alpha)}.
\end{split}
\end{equation*}
If $|s| \leq M_k^{10(2+\alpha)}$, we use the estimate for $(g_l \cdots g_{k+1}g_{k-1} \cdots g_1 F_0)^\wedge(s)$. Since $M_k^{10(2+\alpha)} \leq \min (M_{k+1}, M_{k+1}^{1+\beta})/2$, we use \eqref{mukleq} to obtain
\begin{equation*}
\begin{split}
|(g_l \cdots g_{k+1}\mu_{k-1})^\wedge(s)-\widehat{\mu}_{k-1}(s)| &\leq  \sum_{j=k+1}^l \log(M_j)M_j^{-96}\\
&\lesssim M_{k+1}^{-90}.
\end{split}
\end{equation*}
By Lemma \ref{mukFIhat} with $N=0$,
\begin{equation*}
\begin{split}
\sum_{ |s| \leq M_k^{10(2+\alpha)}} &|\widehat{g_k F_I}(s) ((g_l \cdots g_{k+1} \mu_{k-1} )^\wedge (s) - \widehat{\mu_{k-1}}(s))|\\
&\lesssim \log(M_k)M_k^{-2+10(2+\alpha)}(1+M_k^{-\beta})M_{k+1}^{-90} \leq M_k^{-80(2+\alpha)} .
\end{split}
\end{equation*}
Combining the two cases, we have $\mu_l(F_I) \lesssim \mu_k(F_I) +M_k^{-80(2+\alpha)}$. Since $F_I$ is supported on $2I$, we have $\mu_k(F_{I}) \lesssim \mu_k(2I)$. Thus, we get
\begin{equation}\label{mu_reg_lk}
\mu_l(F_I)  \lesssim \mu_k(2I) +M_k^{-80(2+\alpha)}.
\end{equation}
Because $F_I$ is a continuous function and $\mu_l$ converges weakly to $\mu$, we conclude that $\mu(F_I) \lesssim \mu_k(2I) + M_k^{-80(2 + \alpha)}$. Furthermore, since $1 \lesssim F_I$ on $I$, we conclude that $\mu(I) \lesssim \mu(F_I)$. So we can conclude 
\begin{equation}\label{mu_reg_limit}
\mu(I) \lesssim \mu_k(2I) + M_k^{-80(2+\alpha)}.
\end{equation}
Combining \eqref{mu_reg_limit} with \eqref{mu_k_kint_epsilon}, \eqref{mu_k_kint} and \eqref{mu_k_kint_2} implies the estimates \eqref{muI_kint_1}, \eqref{muI_kint_2} and  \eqref{muI_kint_3} respectively.

\end{proof}
Now we are ready to prove the regularity of the measure $\mu$.
Recall that $\mu$ is the weak limit of a subsequence of the measures $\mu_k$.

\begin{prop}\label{3_lem_mu_reg}
For any interval $I \subseteq \bT$, if $0 \leq \beta <1$, 
\begin{equation*}
\mu(I) \lesssim \log^3(|I|^{-1}) |I|^{\frac{2}{2+\alpha}}.
\end{equation*}
If $-1 < \beta <0$, 
\begin{equation*}
\mu(I) \lesssim \log^3(|I|^{-1}) |I|^{\frac{2+\beta}{2+\alpha}}.
\end{equation*}
\end{prop}
\begin{proof}
Assume that $I$ is an arbitrary interval satisfying 
\[M_k^{-(2 + \alpha)} \leq |I| \leq M_{k-1}^{-(2 + \alpha)}.\]
We will estimate $\mu(I)$ by counting the number of $k$-intervals that intersect $I$. We have the decomposition
\[\mu(I) \leq \sum_{J \in \mathcal{J}_1} \mu(J) + \sum_{J \in \mathcal{J}_2} \mu(J).\]
Observe that the sum over $J \in \mathcal{J}_2$ is empty if $k$ is even. 

We will first estimate the sum over those $J \in \mathcal{J}_1$. The number of $k$-intervals $J \in \mathcal{J}_1$ intersecting $I$ can be estimated in a manner similar to that of Papadimitropoulos \cite{pa-thesis}. We will count the rational numbers $\frac{r}{p}$ contained in the interval $3I$ where $p \in P_{M_j}$. Because each $k$-interval intersecting $I$ must have its center contained in $3I$, this will give an upper bound on the number of $k$-intervals intersecting $I$.

We will split into two cases based on whether the interval $I$ is larger than $M_k^{-1}$. 
\paragraph*{Case 1: $M_k^{-(2 + \alpha)} \leq |I| \leq M_k^{-1}$} Since the rational numbers in $\{r/p : p \in P_{M_k}\}$ are easily seen to be $M_k^{-2}$-separated, the number of rational numbers of this form contained in $3I$ is bounded above by $\max(3|I| M_k^{2}, 1)$. Since each such $k$ interval has $\mu$-measure at most a constant times $\log^3(M_k) M_k^{-2}$ by \eqref{muI_kint_1}, we obtain the following bound:
\begin{equation}\label{muIest2}
\sum_{J \in \mathcal{J}_1}\mu(J) \lesssim \log^3(M_k) \max(  |I|, M_k^{-2}).
\end{equation}
We will now show that this implies the desired regularity estimate for $\sum_{J \in \mathcal{J}_1} \mu(J)$. If the maximum on the right side of \eqref{muIest2} is realized by $\log^3(M_k)|I|$, then we use the inequality $|I| \leq M_k^{-1}$ and conclude
\begin{equation*}
\sum_{J \in \mathcal{J}_1}\mu(J) \lesssim\log^3(|I|^{-1})|I| \leq \log^3(|I|^{-1})|I|^{\frac{2}{2+\alpha}}.
\end{equation*}
If the maximum in \eqref{muIest2} is realized by $\log^3(M_k)M_k^{-2}$, we use the inequality $ M_k^{-(2+\alpha)} \leq |I| $ to conclude
\begin{equation*}
\sum_{J \in \mathcal{J}_1}\mu(J) \lesssim \log^3(|I|^{-1})|I|^{\frac{2}{2+\alpha}}.    
\end{equation*}
This completes the proof of the regularity estimate for $\sum_{J \in \mathcal{J}_1}\mu(J)$ when $|I|$ is small.
\paragraph*{Case 2: $M_k^{-1} \leq |I| \leq M_{k-1}^{-(2 + \alpha)}$} We will count the number of rational numbers of the form $\{r/p : p \in P_{M_k}\}$ contained in $3I$ in a different way. For some fixed $p_0 \in P_{M_k}$, the total number of rational numbers $r/p_0$ contained in $3I$ is at most $\lesssim M_k |I|$. We get that the total number of rational numbers $\{r/p : p \in P_{M_k}\}$ is bounded above by a constant times $|I| |P_{M_k}| M_k$. Using \eqref{muI_kint_2} and $|P_{M_k}| \sim \frac{M_k}{\log M_k}$, we obtain the following bound:
\begin{equation}\label{muIest3}
\begin{split}
\sum_{J \in \mathcal{J}_1}\mu(J) &\lesssim \log^3(M_{k-1}) M_{k-1}^\alpha (2 + M_{k-1}^{-\beta}) |P_{M_k}| \log (M_k) M_k^{-1} |I|\\
&\lesssim \log^3(M_{k-1}) M_{k-1}^{\alpha} (2 + M_{k-1}^{-\beta}) |I|.
\end{split}
\end{equation}
Combining \eqref{muIest3} with the inequality $ |I| \leq M_{k-1}^{-(2 + \alpha)}$, we conclude
\begin{equation*}
\sum_{J \in \mathcal{J}_1}\mu(J) \lesssim \log^3(|I|^{-1}) (|I|^{\frac{2}{2 + \alpha}} + |I|^{\frac{2 + \beta}{2 + \alpha}}) .    
\end{equation*}

It remains to control the contribution of those $J \in \mathcal{J}_2$. Note that this step is only necessary if $k$ is odd. If a $k$-interval $J \in \mathcal{J}_2$ intersects $I$, it follows that the subinterval from $C(\beta, k)$ which intersects $2J$ will be contained in $6I$ and for each subinterval from $C(\beta,k)$, it can intersect at most $2$ $k$-intervals $J \in \mathcal{J}_2$. So, it is sufficient to count the number of intervals $C(\beta, k)$ in $6I$. 

To count the number of intervals of $C(\beta, k)$ in $6I$, we use the $M^{-(1 + \beta)}$-separation of these intervals. This separation implies that the number of such intervals in $6I$ is at most a constant times $\max( M_k^{1 + \beta}|I|, 1)$. Since the $\mu$-measure of each $k$-interval is bounded above by a constant times $\log^3(M_k) M_k^{-2}(2+M_k^{-\beta})$ by \eqref{muI_kint_3}, we obtain
\begin{equation}\label{J_2sum}
\begin{split}
    \sum_{J \in \mathcal{J}_2} \mu(J) &\lesssim \log^3(M_k)M_k^{-2}(2+M_k^{-\beta}) \max( M_k^{1 + \beta}|I|, 1).
\end{split}
\end{equation}
If the maximum on the right side of \eqref{J_2sum} is realized by $M_K^{1+\beta}|I|$, by the inequality $M_k^{-(2+\alpha)} \leq |I|$, we get
\begin{equation*}
\begin{split}
    \sum_{J \in \mathcal{J}_2} \mu(J) &\lesssim \log^3(|I|^{-1}) (|I|^{\frac{3+\alpha-\beta}{2+\alpha} } + |I|^{\frac{3+\alpha}{2+\alpha}})\\
    &\lesssim \log^3(|I|^{-1})(|I|^{\frac{2}{2+\alpha}} + |I|^{\frac{2+\beta}{2+\alpha}}).
\end{split}
\end{equation*}
If the maximum on the right side of \eqref{J_2sum} is realized by $1$, by the inequality $M_k^{-(2+\alpha)} \leq |I|$, we conclude
\begin{equation*}
    \sum_{J \in \mathcal{J}_2} \mu(J) \lesssim \log^3(|I|^{-1}) (|I|^{\frac{2}{2+\alpha} } + |I|^{\frac{2+\beta}{2+\alpha}}).
\end{equation*}

Hence $\mu(I) \lesssim \log^3(|I|^{-1}) |I|^{\frac{2}{2+\alpha}}$ if $0 \leq \beta <1$ and $\mu(I) \lesssim \log^3(|I|^{-1}) |I|^{\frac{2+\beta}{2+\alpha}}$ if $-1 < \beta < 0$, as desired.
\end{proof}

\section{Failure of the extension estimate}\label{failure_sec}

In this section, we define a sequence of functions $f_k$ which witness the failure of the estimate (R). More specifically, we shall prove \eqref{LqLpest2} if $\beta \geq 0$ and \eqref{LqLpest2_2} if $\beta < 0$.  
This will complete the proofs of Theorem \ref{main-result} and Theorem \ref{main-result2}. 

\subsection{Definitions and Lemmas}

We define a sequence of functions $\{f_k\}_{\substack{k \in \bN \\ k \text{ odd}}}$ by
\begin{equation*}
f_k(x) = \sum_{|a| \leq (B_k-1)/2} F(M_k^{\alpha-\beta+1}(B_kx-a)).
\end{equation*}
To obtain a lower bound of $\widehat{f_k \mu_l} (s)$ for $l >k$, we write
\begin{equation*}
\widehat{f_k\mu_l}(s) = \sum_{t \in \bZ} \widehat{f_kg_k}(s-t) (g_l \cdots g_{k+1}\mu_{k-1})^\wedge(t)
\end{equation*}
and split $f_kg_k$ into two functions. Let us consider
\begin{equation*}
h_1(x) = \frac{M_k^{1+\alpha}(1+M_k^{-\beta})}{|P_{M_k}|+M_k^\beta+1} \sum_{|a| \leq (B_k-1)/2} F^2(M_k^{\alpha-\beta+1}(B_kx-a)) 
\end{equation*}
and
\begin{equation*}
h_2(x) =  \frac{M_k^{1+\alpha}}{|P_{M_k}|+M_k^\beta+1} \sum_{p \in P_{M_k}} \sum_{|r| \leq (p-1)/2} F(M_k^{1+\alpha}(px-r))\sum_{|a| \leq (B_k-1)/2}F(M_k^{\alpha-\beta+1}(B_kx-a)).
\end{equation*}
Then,  $g_kf_k = h_1+h_2$ and we have
\begin{equation*}
\widehat{h}_1(s) = \frac{1+M_k^\beta}{|P_{M_k}| +M_k^\beta + 1} \widehat{F^2} \left( \frac{s}{B_kM_k^{1+\alpha-\beta}} \right) \mathbf{1}_{B_k\bZ}(s)
\end{equation*}
and
\begin{equation*}
\widehat{h}_2(s) = \frac{M_k^{\beta-\alpha-1}}{|P_{M_k}| +M_k^\beta + 1} \sum_{p \in P_{M_k}} \sum_{t \in \bZ} \widehat{F} \left( \frac{t}{B_kM_k^{1+\alpha-\beta}} \right)\widehat{F} \left( \frac{s-t}{pM_k^{1+\alpha}}\right) \mathbf{1}_{B_k\bZ}(t) \mathbf{1}_{p\bZ}(s-t).
\end{equation*}

We will use the following estimate on $\widehat{h}_2(s)$.
\begin{lem}\label{FDh_2}
    If $\beta \neq 0$,
    \begin{equation}\label{FDh_2eq1}
        |\widehat{h}_2(s)| \lesssim \frac{1+M_k^{\beta-\alpha}}{M_k} \left(1 + \frac{|s|}{M_k^{2+\alpha}} \right)^{-10}.
    \end{equation}
\end{lem}
\begin{proof}

For fixed $s$, if $|s -t| \geq |s| /2$,
\begin{equation*}
    \left| \widehat{F} \left( \frac{s-t}{p M_k^{1+\alpha}} \right)\right| \lesssim \left( 1+\frac{|s|}{M_k^{2+\alpha}} \right)^{-10},
\end{equation*}
and because $F$ is a Schwartz function, we also have that 
\begin{equation*}
\left|\widehat{F} \left(\frac{t}{B_k M_k^{1 + \alpha - \beta}} \right) \right| \lesssim \left(1 + \frac{|t|}{M_k^{2 + \alpha}} \right)^{-10}.
\end{equation*}
If $|s-t| \leq |s|/2$, then $|s| \sim |t|$. Thus, 
\begin{equation*}
    \left| \widehat{F} \left( \frac{t}{B_kM_k^{1+\alpha-\beta}} \right)\right| \lesssim \left( 1+\frac{|s|}{M_k^{2+\alpha}} \right)^{-10}\left( 1+\frac{|t|}{M_k^{2+\alpha}} \right)^{-10}.
\end{equation*}
and we have the trivial bound
\begin{equation*}
    \left| \widehat{F} \left(\frac{s - t}{p M_k^{1 + \alpha}} \right) \right| \lesssim 1.
\end{equation*}
Therefore, we have for all $t \in \mathbb{Z}$ that 
\begin{equation*}
    \left|\widehat{F} \left(\frac{t}{B_k M_k^{1 + \alpha - \beta}} \right) \widehat{F} \left(\frac{s - t}{p M_k^{1 + \alpha}} \right) \right| \lesssim \left(1 + \frac{|s|}{M_k^{2 + \alpha}} \right)^{-10} \left(1 + \frac{|t|}{M_k^{2 + \alpha}} \right)^{-10}.
\end{equation*}
Thus,
\begin{equation*}
|\widehat{h}_2(s) | \lesssim  \frac{M_k^{\beta-\alpha-1}}{|P_{M_k}| }\left( 1+\frac{|s|}{M_k^{2+\alpha}} \right)^{-10}\sum_{p \in P_{M_k}}\sum_{t \in \bZ}\left( 1+\frac{|t|}{M_k^{2+\alpha}} \right)^{-10}   \mathbf{1}_{B_k\bZ}(t) \mathbf{1}_{p\bZ}(s-t).
\end{equation*}
Assume that $\mathbf{1}_{B_k\bZ}(t) \mathbf{1}_{p\bZ}(s-t) \neq 0$. Then $t$ satisfies the congruences 
\begin{align*}
t & \equiv 0 \quad \text{(mod $B_k$)}, \\
t & \equiv s \quad \text{(mod $p$)}.
\end{align*}
Moreover, we have assumed that $B_k$ is relatively prime to $p$. Therefore, the Chinese remainder theorem implies that $t$ satisfies a congruence of the form 
\[t \equiv a_{s,p} B_k \text{(mod $B_k p$)}\]
for some $0 \leq a_{s,p} < p$. Thus $t = B_k p m + B_k a_{s,p}$ for some $m \in \mathbb{Z}$. Thus,
\begin{equation*}
\begin{split}
|\widehat{h}_2(s) | &\lesssim  \frac{M_k^{\beta-\alpha-1}}{|P_{M_k}| }\left( 1+\frac{|s|}{M_k^{2+\alpha}} \right)^{-10}\sum_{p \in P_{M_k}}\sum_{m \in \bZ}\left( 1+\frac{|pB_km + B_k a_{s,p}|}{M_k^{2+\alpha}} \right)^{-10}\\
& \lesssim  \frac{M_k^{\beta-\alpha-1}}{|P_{M_k}| }\left( 1+\frac{|s|}{M_k^{2+\alpha}} \right)^{-10}\sum_{p \in P_{M_k}}\sum_{m \in \bZ}\left( 1+\frac{|m + a_{s,p}/p|}{M_k^{\alpha-\beta}} \right)^{-10}.
\end{split}
\end{equation*}
It is easy to show that
$$
\sum_{m \in \bZ}\left( 1+\frac{|m + a_{s,p}/p|}{M_k^{\alpha-\beta}} \right)^{-10} \lesssim \begin{cases}
    M_k^{\alpha-\beta}& \text{if } \beta \leq \alpha, \\
    1& \text{if } \alpha <\beta.
\end{cases}
$$
Therefore, we obtain 
$$
|\widehat{h}_2(s) | \lesssim \frac{1+M_k^{\beta-\alpha}}{M_k}\left( 1+\frac{|s|}{M_k^{2+\alpha}} \right)^{-10}.
$$
\end{proof}

\begin{lem}\label{Lpfkmul}
Let $l > k$. For $|m| \leq M_k^{1+\alpha-\beta}/C_F$, we have

\begin{equation}\label{fkmulBksum}
    |\widehat{f_k \mu_l} (B_km)| \gtrsim \log(M_k)M_k^{-1}(1+M_k^\beta).
\end{equation}

\end{lem}
\begin{proof}

We claim that the following estimates hold for sufficiently large $k$ and for $|m| \leq M_k^{1+\alpha-\beta}/C_F$:
\begin{equation}\label{h_1zero}
|\widehat{h}_1(B_km)(g_l \cdots g_{k+1}\mu_{k-1})^\wedge(0)| \gtrsim \log(M_k)M_k^{-1}(1+M_k^\beta);
\end{equation}
\begin{equation}\label{h_1nonzero}
\sum_{ t \neq 0}|\widehat{h}_1(B_km-t)(g_l \cdots g_{k+1}\mu_{k-1})^\wedge(t)| \lesssim \log^3(M_k)M_k^{-99}(1+M_k^\beta);
\end{equation}
\begin{equation}\label{h_2all_1}
\sum_{ t \in \bZ }|\widehat{h}_2(B_km-t)(g_l \cdots g_{k+1}\mu_{k-1})^\wedge(t)| \lesssim \log^{1/10}(M_k)M_k^{-1}(1+ M_k^{\beta-\alpha}) \qquad \text{if} \  \beta  \neq 0.
\end{equation}

Recall that $\alpha >0$. By combining \eqref{h_1zero}, \eqref{h_1nonzero} and \eqref{h_2all_1}, we get \eqref{fkmulBksum} when $\beta \neq 0$.

If $\beta=0$, it suffices to consider $h_1$ only. Since $\beta=0$,  let $B_k=p \in P_{M_k}$. We obtain 
\begin{equation*}
\begin{split}
h_2 =&  \frac{M_k^{1+\alpha}}{|P_{M_k}|+2} \sum_{|r| \leq (p-1)/2} F^2(M_k^{1+\alpha}(px-r))\\
&+\frac{M_k^{1+\alpha}}{|P_{M_k}|+2} \sum_{\substack{p' \in P_{M_k}\\ p' \neq p}} \sum_{|r| \leq (p'-1)/2} F(M_k^{1+\alpha}(p'x-r))\sum_{|a| \leq (p-1)/2}F(M_k^{1+\alpha}(px-a)).    
\end{split}
\end{equation*}
Let $h_{2,1}$ and $h_{2,2}$ denote the first term and the second term respectively. Then, $h_{2,1}$ is a bounded multiple of $h_1$. The function $h_{2,2}$ is nonzero only when $|x| \lesssim M_k^{-(2+\alpha)}$ in $\bT$ since rational numbers in $\{r/p', a/p : p' \in P_{M_k}, p \neq p'\}$ are $M_k^{-2}$-separated. However, $\mu_{k-1}$ is supported away from the origin since $F_0=0$ if $|x| \leq 1/8$. Therefore, $h_{2,2}\cdot (g_l \cdots g_{k+1}\mu_{k-1}) =0$. By \eqref{h_1zero} and \eqref{h_1nonzero}, we obtain \eqref{fkmulBksum} when $\beta=0$.

\paragraph{\textbf{Proof of \eqref{h_1zero}}}
By our choice of the constant $C_F$, the assumption $|m| \leq M_k^{1+\alpha-\beta}/C_F$ implies 
\begin{equation*}
|\widehat{h}_1(B_km)| \gtrsim \log(M_k)M_k^{-1}(1+M_k^\beta).
\end{equation*}
By iterating \eqref{convgMres1} and \eqref{mukleq}, we get
\begin{equation}\label{diffF_00}
|(g_l \cdots g_{k+1}\mu_{k-1})^\wedge(0) -\widehat{F}_0(0)| \leq \sum_{i=1}^l \log(M_i)M_i^{-96} \leq M_1^{-90}.
\end{equation}
Since $\widehat{F}_0(0)=1$, we have $|(g_l \cdots g_{k+1}g_{k-1} \cdots g_1F_0)^\wedge(0)| \gtrsim 1$. Thus, \eqref{h_1zero} follows.

\paragraph{\textbf{Proof of \eqref{h_1nonzero}}} The proof is similar to that of the estimate on $\mu$ for $k$-intervals, but we also use the properties of $\widehat{h}_1$.

If $|t| \geq M_k^{10(2+\alpha)}$, $|B_km-t| \geq |t|/2$ holds since $|B_km| \lesssim M_k^{2+\alpha}$. Thus, we get
\begin{equation*}
|\widehat{h}_1(B_km-t)| \lesssim \log(M_k)M_k^{-1}(1+M_k^\beta) \left(\frac{M_k^{2+\alpha}}{|t|} \right)^{10}.
\end{equation*}
By \eqref{mukleq2}, we obtain 
\begin{equation*}
\sum_{ |t| \geq M_k^{10(2+\alpha)}}|\widehat{h}_1(B_km-t)(g_l \cdots g_{k+1}\mu_{k-1})^\wedge(t)| \lesssim \log(M_k)(1+M_k^\beta)M_k^{-1-80(2+\alpha)}.
\end{equation*}
If $|t| \leq M_k^{10(2+\alpha)}$, by iterating \eqref{mukleq}, we have
\begin{equation*}
|(g_l \cdots g_{k+1} \mu_{k-1})^\wedge(t) | \lesssim |\widehat{\mu}_{k-1}(t)| + M_{k+1}^{-90}.
\end{equation*}
Also, observe that $t=B_km$ for some nonzero $m$ since $\widehat{h}_1$ is supported on $B_k\bZ$ and $t \neq 0$. Thus, $|t| \geq M_k^{1+\beta} \geq 2M_{k-1}^{2+\alpha}$, and \eqref{convgMres3} with $N=N(\beta)$ together with \eqref{Mcond4} imply that
\begin{equation*}
\begin{split}
|\widehat{\mu}_{k-1}(t)|&\leq \log(M_{k-1}) M_{k-1}^{-1+(2+\alpha)(N(\beta)+1)}(1+M_{k-1}^\beta)\log(|t|)|t|^{-N(\beta)}\\
&\lesssim \log(M_k)\log(|t|)|t|^{-N(\beta)}.   
\end{split}
\end{equation*}
Using $|\widehat{h}_1(s)| \lesssim \log(M_k)M_k^{-1}(1+M_k^\beta)$ and \eqref{mukleq}, we obtain
\begin{equation*}
\begin{split}
\sum_{0< |t| \leq M_k^{10(2+\alpha)}}&|\widehat{h}_1(B_km-t)(g_l \cdots g_{k+1}\mu_{k-1})^\wedge(t)|\\
&\lesssim \log(M_k)M_k^{-1}(1+M_k^\beta) \left(M_k^{10(2+\alpha)}M_{k+1}^{-90} + \log(M_k)\sum_{|t| \geq M_k^{1+\beta}}\log(|t|)|t|^{-N(\beta)} \right)\\
&\lesssim \log^3(M_k)M_k^{-99}(1+M_k^\beta).
\end{split}
\end{equation*}
In the first inequality, we used that $|t| \geq M_k^{1+\beta}$ because of the support of $\widehat{h}_1$. Hence, we get \eqref{h_1nonzero}.

\paragraph{\textbf{Proof of \eqref{h_2all_1}}}

By the same method we used in the proof of \eqref{h_1nonzero}, we use Lemma \ref{FDh_2} and obtain
\begin{equation*}
        \sum_{|t| \geq M_k^{10(2+\alpha)}}|\widehat{h}_2(B_km-t)(g_l \cdots g_{k+1}\mu_{k-1})^\wedge(t)|\lesssim (1+ M_k^{\beta-\alpha})M_k^{-1-80(2+\alpha)}.
\end{equation*}
By \eqref{Mcond4}, we have $\log(M_k)^{1/10} \geq 2M_{k-1}^{2+\alpha}$. By Lemma \ref{FDh_2}, $|\widehat{h}_2(s)| \lesssim M_k^{-1}(1+M_k^{\beta-\alpha})$. Thus, we use \eqref{convgMres3} with $N=99$ and \eqref{mukleq} to obtain
\begin{equation*}
    \begin{split}
        \sum_{\log^{1/10}(M_k) \leq |t| \leq M_k^{10(2+\alpha)}}&|\widehat{h}_2(B_km-t)(g_l \cdots g_{k+1}\mu_{k-1})^\wedge(t)| \nonumber \\
&\lesssim  M_k^{-1}(1+M_k^{\beta-\alpha}) \left(M_k^{10(2+\alpha)}M_{k+1}^{-90} + \log(M_k)\sum_{|t| \geq \log^{1/10}(M_k)}\log(|t|)|t|^{-99} \right) \nonumber\\
&\lesssim \log^{-5}(M_k)M_k^{-1} (1+M_k^{\beta-\alpha}).    
    \end{split}
\end{equation*}
By Lemma \ref{FDh_2} and \eqref{mukleq2}, we get
\begin{equation}\label{h_2all_eq}
\sum_{|t| \leq \log(M_k)^{1/10}}|\widehat{h}_2(B_km-t)(g_l \cdots g_{k+1}\mu_{k-1})^\wedge(t)| \lesssim \log^{1/10}(M_k)M_k^{-1} (1+M_k^{\beta-\alpha}).
\end{equation}
Combining the three cases, we get \eqref{h_2all_1}.

\end{proof}

\subsection{Proof of \texorpdfstring{\eqref{LqLpest2}}{LqLpest2}  and \texorpdfstring{\eqref{LqLpest2_2}}{LqLpest2-2} } 

Now we complete the proofs of Theorem \ref{main-result} and Theorem \ref{main-result2} by proving \eqref{LqLpest2} for $\beta \geq 0$ and \eqref{LqLpest2_2} for $\beta < 0$.  

Let us prove \eqref{LqLpest2} first. Assume $\beta \geq 0$. Observe that each $F(M_k^{\alpha-\beta+1}(B_kx-a))$ is supported on a $k$-interval in $\mathcal{J}_2$ and their supports are disjoint. Since $|F| \lesssim 1$ and $\beta \geq 0$,  \eqref{muI_kint_3} implies that
\begin{equation}\label{f_kLq}
\norm{f_k}_{L^q(\mu)}^q \lesssim \log^3(M_k)  M_k^{-2}B_k \lesssim \log^3(M_k)M_k^{\beta-1} \qquad \text{for} \ 1 \leq q < \infty.
\end{equation}
We will estimate $\mathcal{F}(f_k \mu)(\xi+s)$ for $s \in \mathbb{Z}$ and $|\xi| \leq (\log(M_k))^{-3}$. In fact, since $\mu$ is supported on $[-1/2,1/2]$, we obtain
\begin{align*}
\left| \mathcal{F}(f_k \mu)(\xi + s) - \mathcal{F}(f_k \mu)(s) \right|= & \left| \int f_k(x) \left(e^{2 \pi i (\xi + s)x} - e^{2 \pi i s x} \right) \, d \mu(x) \right| \\
\leq & \int |f_k(x)| 2 \pi |x \xi| \, d \mu(x) \\
\lesssim & |\xi| \int |f_k(x)| \, d \mu(x) \\
\lesssim & |\xi| \log^3( M_k) M_k^{\beta - 1}.
\end{align*}
In the last inequality, we used \eqref{f_kLq}. The bound on $\xi$ guarantees that this is much smaller than $\log (M_k) M_k^{\beta - 1} $. By Lemma \ref{Lpfkmul}, for large $k$, we have that $\mathcal{F}(f_k \mu)(B_k m + \xi) \gtrsim \log(M_k) M_k^{\beta - 1} $ whenever $|m| \leq M_k^{1 + \alpha - \beta}/C_F$ . Thus, we obtain
\begin{equation*}
\norm{\mathcal{F}(f_k\mu)}_{L^p(\lambda)}^p \geq \int_{|\xi| \leq (\log M_k)^{-3}} \sum_{s \in \bZ} |\mathcal{F}(f_k\mu)(s+\xi)|^p \gtrsim M_k^{\alpha-\beta+1} \log(M_k)^{-3} (\log(M_k)M_k^{\beta-1})^p.
\end{equation*}
Therefore, we get
\begin{equation*}
\frac{\norm{\mathcal{F}(f_k\mu)}_{L^p(\lambda)}}{\norm{f_k}_{L^q(\mu)}} \gtrsim (\log(M_k))^{1-\frac{3}{p}-\frac{3}{q}} M_k^{ (1 - \beta)(\frac{1}{q}-1)+\frac{\alpha-\beta+1}{p}}.
\end{equation*}
If $p < p_+(q)$, the right hand side diverges to $\infty$ as $k \rightarrow \infty$. We have established \eqref{LqLpest2}.

The proof of \eqref{LqLpest2_2} is similar. Since $\beta <0$, instead of \eqref{f_kLq}, we have
\begin{equation*}
\norm{f_k}_{L^q(\mu)}^q \lesssim \log^3(M_k)  M_k^{-2-\beta}B_k \lesssim \log^3(M_k)M_k^{-1} \qquad \text{for} \ 1 \leq q < \infty.
\end{equation*}
We use Lemma \ref{Lpfkmul} when $-1<\beta<0$ and we obtain
\begin{equation*}
\norm{\mathcal{F}(f_k\mu)}_{L^p(\lambda)}^p \gtrsim M_k^{\alpha-\beta+1} \log(M_k)^{-3} (\log(M_k)M_k^{-1})^p.
\end{equation*}
Thus,
\begin{equation*}
\frac{\norm{\mathcal{F}(f_k\mu)}_{L^p(\lambda)}}{\norm{f_k}_{L^q(\mu)}} \gtrsim (\log(M_k))^{1-\frac{3}{p}-\frac{3}{q}} M_k^{(\frac{1}{q}-1)+\frac{\alpha-\beta+1}{p}}.
\end{equation*}
If $p < p_-(q)$, the right hand side diverges to $\infty$ as $k \rightarrow \infty$. We have established \eqref{LqLpest2_2}.

\section*{Declarations}
\subsection*{Funding}
The authors did not receive support from any organization for the submitted work.
\subsection*{Competing Interests}
On behalf of all authors, the corresponding author states that there is no conflict of interest.
\subsection*{Data Availability Statement}
We do not analyse or generate any datasets, because our work is a mathematical proof. 
\bibliographystyle{myplainurl}
\bibliography{refs}

\end{document}